\documentclass{amsart}
\usepackage{amsmath,amssymb,amsthm,amscd}
\usepackage[a4paper,text={140mm,240mm},centering,headsep=5mm,footskip=10mm]{geometry}
\usepackage[matrix,arrow]{xy}
\numberwithin{equation}{section}

\theoremstyle{plain}
\newtheorem{theo}{Theorem}[section]
\newtheorem{lem}[theo]{Lemma}

\newtheorem{prop}[theo]{Proposition}
\newtheorem{conj}[theo]{Conjecture}
\newtheorem{coro}[theo]{Corollary}
\newtheorem{claim}[theo]{Claim}

\theoremstyle{definition}
\newtheorem{defn}[theo]{Definition}
\newtheorem{defi}[theo]{Definition}
\newtheorem{const}[theo]{Construction}

\theoremstyle{remark}
\newtheorem{rem}[theo]{Remark}

\newtheorem{exam}[theo]{Example}

\newcommand{\Mo}{\mathcal{M}}
\newcommand{\Chow}{{\it Chow}}
\newcommand{\Fix}{{\rm Fix}}

\newcommand{\eq}{{\rm eq}}

\newcommand{\Spec}{\mathrm{Spec}}

\newcommand{\Aut}{\text{\rm Aut}}
\newcommand{\Hom}{\text{\rm Hom}}

\newcommand{\Q}{{\mathbb Q}}
\newcommand{\Z}{{\mathbb Z}}

\newcommand{\CH}{\mathrm{CH}}

\newcommand{\cone}{\mathrm{cone}}
\newcommand{\Tot}{\mathrm{Tot}}
\newcommand{\SL}{\mathrm{SL}}

\newcommand\et{{\text{\'{e}t}}}    

\newcommand\ch{{\mathrm{ch}}}    

\newcommand\cF{{\mathcal F}}

\newcommand\cC{{\mathcal C}}

\newcommand\cS{{\mathcal S}}

\newcommand\Lam{\Lambda}
\newcommand\Ln{\Lam_n}

\renewcommand\k{\kappa}

\newcommand{\wh}[1]{\widehat{#1}}

\newcommand{\ol}[1]{\overline{#1}}

\newcommand{\sumd}[2]{\underset{x\in {#1}_{(#2)}}\bigoplus}
\newcommand{\sumdy}[2]{\underset{y\in {#1}_{(#2)}}\bigoplus}

\newcommand{\sumdG}[2]{\underset{x\in {#1}_{(#2)}/G}\bigoplus}

\newcommand\bC{{\mathbb C}}
\newcommand\bR{{\mathbb R}}
\newcommand\bZ{{\mathbb Z}}
\newcommand\bQ{{\mathbb Q}}

\newcommand\bP{{\mathbb P}}

\newcommand\qz{{\bQ}/{\bZ}}
\newcommand\qzl{{\bQ_\ell}/{\bZ_\ell}}

\newcommand\nz{\bZ/n\bZ}

\newcommand\lnz{\bZ/\ell^n\bZ}

\newcommand{\indlim}[1]{\underset{{\underset{#1}{\longrightarrow}}}{\mathrm{lim}}\; }

\newcommand{\rmapo}[1]{\overset{#1}{\longrightarrow}}

\newcommand{\lmapo}[1]{\overset{#1}{\longleftarrow}}
\newcommand\isom{\overset{\cong}{\longrightarrow}}

\newcommand\qfor{\quad\text{for }}

\newcommand\Sb{\ol S}

\newcommand\Xb{\ol X}
\newcommand\Yb{\ol Y}

\newcommand\fb{\ol f}
\newcommand\pib{\ol \pi}

\newcommand\tX{\widetilde{X}}
\newcommand\tY{\widetilde{Y}}

\newcommand\Het{H^{\et}}

\newcommand{\graphhom}[2]{\gamma^{#1}_{#2}}

\newcommand\eqRS{{\text{\bf (RS)}_{eq}}}

\newcommand{\Xd}[1]{X_{(#1)}}
\newcommand{\Xcd}[1]{X^{(#1)}}

\newcommand{\Yd}[1]{Y_{(#1)}}
\newcommand{\Ud}[1]{U_{(#1)}}

\newcommand\Hc{H_{c}}

\newcommand{\HcG}[1]{\Hc^{#1}}
\newcommand{\HcGX}[2]{\Hc^{#1}(X,G;#2)}
\newcommand{\HcGY}[2]{\Hc^{#1}(Y,G;#2)}
\newcommand{\HcGU}[2]{\Hc^{#1}(U,G;#2)}
\newcommand{\HcGZ}[2]{\Hc^{#1}(Z,G;#2)}
\newcommand{\HcGx}[1]{\Hc^{#1}}
\newcommand{\HcHx}[1]{\Hc^{#1}}

\newcommand{\HG}[2]{\Het_{#1}(#2,G)}

\newcommand{\HWGM}[2]{H^W_{#1}(#2,G;M)}
\newcommand{\HWM}[2]{H^W_{#1}(#2;M)}

\newcommand{\HWG}[2]{H^W_{#1}(#2,G)}
\newcommand{\HWGL}[2]{H^W_{#1}(#2,G;\Lambda)}
\newcommand{\HWGLinf}[2]{H^W_{#1}(#2,G;\Linf)}
\newcommand{\HWGX}[1]{H^W_{#1}(X,G)}
\newcommand{\HWGXa}[1]{H^W_{#1}(X_a,G)}

\newcommand{\HWGQ}[2]{H^W_{#1}(#2,G;\bQ)}

\newcommand{\HWL}[2]{H^W_{#1}(#2;\Lam)}

\newcommand{\HWQ}[2]{H^W_{#1}(#2;\bQ)}
\newcommand{\HWLinf}[2]{H^W_{#1}(#2;\Linf)}

\newcommand\Shv{{\it Shv}}
\newcommand\ShGXL{{\it Shv}_G(X,\Lambda)}
\newcommand\ShGYL{{\it Shv}_G(Y,\Lambda)}
\newcommand{\ShGL}[1]{{\it Shv}_G(#1,\Lambda)}

\newcommand\ModLG{\Mod_{\Lambda[G]}}
\newcommand\ModL{Mod_{\Lambda}}

\newcommand\sch{\cC_{/k}}

\newcommand\Sch{\cC}
\newcommand\ChowG{\Chow_{G/k}}

\newcommand\schGkp{\Sch_{G/k'}}
\newcommand\schG{\Sch_{G/k}}

\newcommand\tk{\widetilde{k}}

\newcommand\eqsch{\Sch_{\eq/k}}

\newcommand\eqschp{\Sch_{\eq/k *}}
\newcommand\eqscht{\Sch_{\eq/\tk}}

\newcommand\cSeq{\cS_{\eq/k}}
\newcommand\cSeqp{\cS_{\eq/k'}}

\newcommand\cSpr{\cS^{\rm prim}_{/k}}
\newcommand\cSGpr{\cS^{\rm prim}_{G/k}}
\newcommand\cSG{\cS_{G/k}}

\newcommand\LSG{\Lam\cSG}

\renewcommand\Ln{\Lambda_n}

\newcommand\Linfty{\Lambda_\infty}
\newcommand\Linf{\Lambda_\infty}

\newcommand\Mod{{Mod}}

\newcommand{\KCM}[1]{KC(#1;M)}

\newcommand{\KHM}[2]{KH_{#1}(#2;M)}

\newcommand{\KHG}[2]{KH_{#1}(#2,G)}
\newcommand{\KHGX}[1]{KH_{#1}(X,G)}
\newcommand{\KHGXa}[1]{KH_{#1}(X_a,G)}

\newcommand{\KHLinf}[2]{KH_{#1}(#2,\Linf)}

\newcommand{\KHGM}[2]{KH_{#1}(#2,G;M)}
\renewcommand{\KHM}[2]{KH_{#1}(#2;M)}
\newcommand{\KHnz}[2]{KH_{#1}(#2;\nz)}
\newcommand{\KHqzl}[2]{KH_{#1}(#2;\qzl)}
\newcommand{\KHlnz}[2]{KH_{#1}(#2;\lnz)}

\newcommand{\KHGLinf}[2]{KH_{#1}(#2,G;\Linf)}

\newcommand{\KCGM}[1]{KC(#1,G;M)}

\newcommand{\edgehom}[2]{\epsilon^{#1}_{#2}}

\newcommand{\EGXM}[3]{E^{#1}_{#2,#3}(X,G;M)}
\newcommand{\EGX}[3]{E^{#1}_{#2,#3}(X,G)}

\newcommand\EYb{\overline{E}_Y}
\newcommand\pihat{\widehat{\pi}}
\renewcommand{\wh}[1]{\widehat{#1}}
\newcommand{\wht}[1]{\widehat{\tilde{#1}}}
\newcommand{\dualG}[1]{\Gamma(#1)}

\newcommand{\pics}[1]{\pi_1^{cs}(#1)}

\begin{document}

\title{
Cohomological Hasse principle and resolution of quotient singularities}

\author{Moritz Kerz}
\author{Shuji Saito}
\address{Moritz Kerz\\
NWF I-Mathematik\\
Universit\"at Regensburg\\
93040 Regensburg\\
Germany}
\email{moritz.kerz@mathematik.uni-regensburg.de}
\address{Shuji Saito\\
Interactive Research Center of Science, 
Graduate School of Science and Engineering,
 Tokyo Institute of Technology\\
Ookayama, Meguro\\
Tokyo 152-8551\\
Japan
}
\email{sshuji@msb.biglobe.ne.jp}

\begin{abstract}
In this paper we study weight homology of singular schemes.
Weight homology is an invariant of a singular scheme defined in terms of hypercoverings of resolution of singularities. 
Our main result is McKay principle for weight homology of quotient singularities, 
i.e.\ we describe weight homology of a quotient scheme in terms of weight homology of an equivariant scheme.
Our method is to reduce the geometric McKay principle for weight homology to Kato's cohomological Hasse principle for arithmetic schemes.
The McKay principle for weight homology implies McKay principle for the homotopy type of the dual complex of 
the exceptional divisors of a resolution of a quotient singularity.
As a consequence we show that the dual complex is contractible for isolated quotient singularities.
\end{abstract}

\maketitle

\tableofcontents

\section*{Introduction}\label{intro}

Quite a few examples have been observed which show that an arithmetic method can play a significant role 
for a geometric question. In this paper we present such a new example.
The geometric question concerns the dual (or configuration) complex of the exceptional divisors
of a resolution of a quotient singularity $X/G$, where $X$ is a quasi-projective smooth
scheme over a perfect field $k$
endowed with an action of a finite group $G$. Let $Z$ be the singular locus of $X/G$.
Assume furthermore that there exists a resolution of singularities $g:\tY \to X/G$ such that $g$ is proper birational and an isomorphism outside $Z$
and $E=g^{-1}(Z)_{red}$ is a simple normal crossing divisor on the smooth scheme $\tY$.
The dual complex $\dualG E$ is a $CW$-complex (which is a $\Delta$-complex in the sense of \cite[Section 2.1]{Hat})
whose $a$-simplices correspond to the connected components of 
\[
E^{[a]} =\underset{1\leq i_0< i_1< \cdots< i_{a}\leq N}{\coprod} 
E_{i_0}\cap\cdots\cap E_{i_a}
\]
where $E_1,\dots, E_N$ are the irreducible components of $E$.
\medbreak

A model case is a Klein quotient singularity $\bC^2/G$, where $G\subset \SL_2(\bC)$ acts linearly on $\bC^2$.
The origin $0\in \bC^2/G$ is the unique singular point of $\bC^2/G$.
Let $g: \tY \to \bC^2/G$ be the minimal resolution.
The irreducible components of the exceptional locus $E=g^{-1}(0)$ are rational curves
that form a configuration expressed by the Dynkin diagrams.
For example, if we take a binary dihedral
\[
G=<\sigma,\tau\;|\; \sigma^n=1,\; \tau^2=-1,\; \tau\sigma\tau=-\sigma^{-1} >,
\]
the irreducible components of $E$ form a configuration looking as

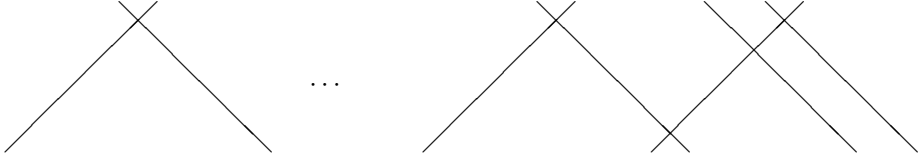
\begin{center}
\setlength{\unitlength}{1mm}
\hskip 50pt
\begin{picture}(200,25)

\put(0,0){\line(1,1){20}}
\put(15,20){\line(1,-1){20}}
\put(40,8){$\cdots$}
\put(55,0){\line(1,1){20}}
\put(70,20){\line(1,-1){20}}
\put(85,0){\line(1,1){20}}
\put(100,20){\line(1,-1){20}}
\put(92,20){\line(1,-1){20}}

\end{picture}
\end{center}
\bigskip
consisting of $(n+2)$ rational curves.
The configuration complex $\dualG E$ is 
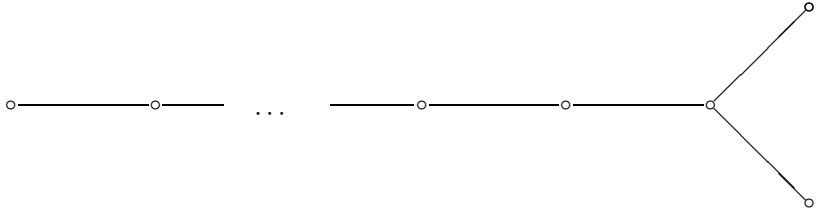
\begin{center}
\setlength{\unitlength}{1mm}
\begin{picture}(200,20)

\put(18,0){\circle{1}}
\put(19,0){\line(1,0){17}}
\put(37,0){\circle{1}}
\put(38,0){\line(1,0){8}}
\put(50,-2){$\cdots$}
\put(60,0){\line(1,0){11}}
\put(72,0){\circle{1}}
\put(73,0){\line(1,0){17}}
\put(91,0){\circle{1}}
\put(92,0){\line(1,0){17}}
\put(110,0){\circle{1}}
\put(110.5,0.5){\line(1,1){12}}
\put(123,13){\circle{1}}
\put(110.5,-0.5){\line(1,-1){12}}
\put(123,-13){\circle{1}}
\put(123,13){\circle{1}}

\end{picture}
\end{center}
\vskip 50pt
with $(n+2)$ vertices and $(n+1)$ edges.
\medbreak

For general $G\subset\SL_2(\bC)$ McKay observed a mysterious coincidence of $\dualG E$ and
the so-called McKay graph which is computed in terms of representations of $G$.
The higher dimensional generalization of this fact, called McKay correspondence, is 
now a fertile land of algebraic geometry. For more details on this, we refer the readers 
to an excellent exposition by M.~Reid \cite{Re}. Here we just quote the following principle:
\medbreak\noindent
{\bf McKay principle:}
Let $G$ be a finite group and $X$ be a quasi-projective smooth $G$-scheme over a field $k$ (i.e.
a quasi-projective smooth scheme over $k$ endowed with an action of $G$), 
and $g:\tY\to X/G$ be a resolution of singularities of $X/G$.
Then the answer to any well posed question about the geometry of $\tY$ is 
the $G$-equivariant geometry of $X$.
\medbreak

In this paper we investigate McKay principle for the homotopy type of $\dualG E$, which is known to be 
an invariant of $X/G$ independent of a choice of a resolution $\tY\to X/G$ by a theorem of Stepanov
and its generalizations (\cite{St}, \cite{ABW}, \cite{Pa}, \cite{Th}).
\bigskip

Let us fix the setup for our main result.
Let $X$ and $G$ be as before and $\pi:X\to X/G$ be the projection.
Fix a closed reduced subscheme $S\subset X/G$ which is projective over $k$ and contains the singular locus $(X/G)_{sing}$ of $X/G$.
Let $T=\pi^{-1}(S)_{red}$ be the reduced part of $\pi^{-1}(S)$. 
Assume that we are given the following datum:
\begin{itemize}
\item
a proper birational morphism $g:\tY \to X/G$ such that $\tY$ is smooth, $E_S=g^{-1}(S)_{red}$ is a simple normal crossing divisor on $\tY$ and $g$ is an isomorphism over $X/G-S$.
\item
a proper birational $G$-equivariant morphism $f:\tX \to X$ in $\schG$ such that $\tX$ is a smooth $G$-scheme, 
$E_T=f^{-1}(T)_{red}$ is a $G$-strict simple normal crossing divisor on $\tX$ 
(cf. Definition \ref{def.Gstrict}) and $f$ is an isomorphism over $X-T$.
\end{itemize}
\[
\begin{CD}
\pi^{-1}(S)_{red}=\; @. T @>>>  X @<{f}<< \tX @<<< E_T = f^{-1}(T)_{red}   \\
 @. @VVV @VV{\pi}V\\
 @. S  @>>> {X/G} @<{g}<<   \tY  @<<< E_S=g^{-1}(S)_{red} 
\end{CD}
\]
\medbreak

Note that we do not assume the existence of a morphism $E_T \to E_S$.
By definition $G$ acts on $\Gamma(E_T)$ and we can form a $CW$-complex $\Gamma(E_T)/G$.

\begin{theo}\label{thm.MPhtintro}(Theorem \ref{thm.MPht}) 
We assume that $\ch(k)=0$ or that $k$ is perfect and canonical resolution of singularities in the sense of \cite{BM} holds over $k$.
In the homotopy category of $CW$-complexes, there exists a canonical map
\[
\phi: \Gamma(E_T)/G \to \Gamma(E_S)
\]
which induces isomorphisms on the homology and fundamental groups:
\[
H_a(\Gamma(E_T)/G) \isom  H_a(\Gamma(E_S))\qfor \forall a\in \bZ,
\]
\[
\pi_1(\Gamma(E_T)/G) \isom  \pi_1(\Gamma(E_S)).
\]
\end{theo}
\medbreak

By using basic theorems in algebraic topology (Whitehead and Hurewicz), Theorem \ref{thm.MPhtintro}
implies the following:

\begin{coro}\label{thm.homequiv.intro} 
Let the assumption be as in Theorem \ref{thm.MPhtintro}.
\begin{itemize}
\item[(1)]
If $\Gamma(E_T)/G$ is simply connected, $\phi$ is a homotopy equivalence.
\item[(2)]
If $\dualG {E_T}/G$ is contractible, $\dualG {E_S}$ is contractible.
\item[(3)]
If $T$ is smooth (e.g. $\dim(T)=0$ which means that $(X/G)_{sing}$ is isolated), 
$\dualG {E_S}$ is contractible.
\end{itemize}
\end{coro}

We will deduce the following variant of Corollary \ref{thm.homequiv.intro}(3).

\begin{coro}(Corollary \ref{thm.homequiv3})
Let $A$ be a complete regular local ring containing $\Q$ and let $G$ be a finite group acting on
$A$. Set $X=\Spec (A)$ and assume that $X/G$ has an isolated singularity $s\in X/G$. Let $g:\tilde Y\to
X/G$ be a proper morphism such that $g$ is an isomorphism outside $s$ and $E_s=g^{-1}(s)_{red}$
is a simple normal crossing divisor in the regular scheme $\tilde Y$. Then the topological
space $\Gamma(E_s)$
is contractible.
\end{coro}

\bigskip

Here we recall some known results on the contractibility of the dual complex of the exceptional divisor of a resolution of 
singularities. Let $(Y,S)$ be an isolated singularity over algebraically closed field of characteristic $0$. 
We recall the following implications for $(Y,S)$:
\[
\text{finite quotient $\Rightarrow$ KLT $\Rightarrow$ rational,}
\]
where $KLT$ stands for Kawamata log terminal.
Let $g: \tY\to Y$ be proper birational such that 
$\tY$ smooth, $E_S=g^{-1}(S)_{red}$ is a simple normal crossing divisor and $g$ is an isomorphism over $Y-S$.
The following facts are known:
\begin{itemize}
\item
If $(Y,S)$ is rational, $H_*(\Gamma(E_S))$ is torsion (\cite{ABW2}).
The proof uses weight argument in Hodge theory.
\item
there exists a rational singularity $(Y,S)$ such that $\Gamma(E_S)$ has the homotopy type of $\bP^2_{\bR}$,
in particular $\pi_1(\Gamma(E_S))= H_1(\Gamma(E_S))=\bZ/2\bZ$ (\cite{Pa}).
\item
If $(Y,S)$ is KLT, $ \pi_1(\Gamma(E_S))=1$. This is a consequence of a theorem of Koll\'ar and Takayama (\cite{Ko}, \cite{Ta}).
\end{itemize}

Recently, the intriguing question  whether $\Gamma(E_S)$ is contractible if $(Y,S)$ is KLT has been
answered positively in \cite{FKX}.
\bigskip

Now we explain the main idea of the proof of Theorem \ref{thm.MPhtintro}.
The proof of the assertion on the fundamental group relies on a geometric interpretation of
the fundamental group of the dual complex of a simple normal crossing divisor.
More generally, for any locally noetherian scheme $E$, we associate a $CW$-complex $\Gamma(E)$ to $E$
in the same way as the case of simple normal crossing divisors and show a natural isomorphism
\[
\pi_1(\Gamma(E))\simeq \pics E,
\]
where the right hand side denotes the classifying group of {\it cs-coverings} $E'\to E$,
which is by definition a morphism of schemes such that any point $x\in E$ has a Zariski open neighborhood $U\subset E$ such that $E'\times_E U \simeq \coprod U$ with a possibly infinite coproduct (see \S\ref{csfundgroup}). 
\bigskip
 
In order to show the assertion of Theorem \ref{thm.MPhtintro} on the homology groups, 
we need introduce the equivariant weight homology. 
Let $\eqsch$ be the category of pairs $(X,G)$ of a finite group $G$ and a $G$-scheme $X$ which is quasi-projective over $k$ (see \S\ref{eqWhom} for the definition of morphisms in $\eqsch$). 
For simplicity we ignore the $p$-torsion if $\ch(k)=p>0$ and work over $\Lam:=\displaystyle{\bZ[\frac{1}{p}]}$
(we can extend the following results to the case $\Lam=\bZ$ assuming appropriate form of resolution of singularities).
Let $\ModL$ be the category of $\Lam$-modules.
An equivariant homology theory $H=\{H_a\}_{a\ge 0}$ with values in $\Lam$-modules on $\eqsch$ is
a sequence of functors:
$$
H_a(-): \eqsch  \rightarrow \ModL\quad (a\in \bZ_{\geq 0})
$$
which are covariant for proper morphisms and contravariant for strict open immersions
(a strict morphism in $\eqsch$ means a $G$-equivariant morphism for some fixed $G$) such that
if $i:Y\hookrightarrow X$ is a strict closed immersion in $\eqsch$,
with open complement $j:V\hookrightarrow X$ there is a long exact sequence
called localization sequence (see Definition \ref{def.eqhom})
$$
\cdots\rmapo{\partial} H_a(Y) \rmapo{i_*} H_a(X) \rmapo{j^*} H_a(V) \rmapo{\partial} H_{a-1}(Y)
\longrightarrow \cdots.
$$

A main input is the following result (see Theorem \ref{thm.Whom}).

\begin{theo}\label{thm.intro.Whom}
For a $\Lam$-module $M$, there exists a homology theory on $\eqsch$ with values in $\Lam$-modules:
\[
(X,G) \mapsto \HWGM {a} X\quad (a\in \bZ_{\geq 0})
\]
called equivariant weight homology with coefficient $M$ satisfying the following condition:
Let $E$ be a projective $G$-scheme over $k$ which is a $G$-strict simple normal crossing divisor 
on a smooth $G$-scheme over $k$ (see Definition \ref{def.Gstrict} for $G$-strict).
Then we have 
\[
\HWGM {*} E \simeq H_*(\Gamma(E)/G)
\]
which is computed as the homology of the complex:
$$ 
\cdots\to M^{\pi_0(E^{[a]})/G} \rmapo{\partial} M^{\pi_0(E^{[a-1]})/G} 
\rmapo{\partial} \cdots \rmapo{\partial} M^{\pi_0(E^{[1]})/G}. 
$$
In particular, if $X$ is a projective smooth $G$-scheme over $k$, we have

$$\HWGM a X \cong \left\{ \begin{array}{lr}
M^{\Xcd 0/G} & \text{ for } a=0 \\
0 & \text{ for } a\ne 0 \end{array} \right. $$ 
where $\Xcd 0$ is the set of the generic points of $X$.
\end{theo}

In case $G$ is the trivial group $e$, we write $\HWM a X$ for $\HWGM a X$.
\medbreak

The construction of $\HWGM * X$ hinges on a descent argument due to
Gillet--Soul\'e (\cite{GS1}, \cite{GS2}) and Jannsen \cite{J1},
going back to Deligne \cite{SGA4} Expos\'e $\rm V^{bis}$. If $X$ is proper over $k$, we take a $\Lam$-admissible hyperenvelope
$(X_\bullet,G)$ (see Definition \ref{def.hypenv}) which is a certain simplicial object in $\eqsch$ such that
$X_n$ are smooth projective over $k$ for all $n$, and then define $\HWGM * X$ as homology groups of the complex
$$ 
\cdots\to M^{\pi_0(X_n)/G} \rmapo{\partial} M^{\pi_0(X_{n-1})/G} 
\rmapo{\partial} \cdots \rmapo{\partial} M^{\pi_0(X_0)/G}. 
$$
The existence of hyperenvelopes relies on equivariant resolution of singularities.
In case $\ch(k)=0$ it follows from canonical resolution of singularities \cite{BM}.
In case $\ch(k)>0$ the construction depends on Gabber's refinement of de Jong's alteration theorem \cite{Il}.
\medbreak

Theorem \ref{thm.MPhtintro} is deduced from the following theorem
(see Theorem \ref{thm.Whomgraph} and  Corollary \ref{cor.Whomgraph2} in \S\ref{eqWhom}).

\begin{theo}[McKay principle for weight homology] \label{MP.intro}
For $(X,G)\in \eqsch$ let $\pi: X\to X/G$ be the projection viewed as a morphism 
$(X,G)\to (X/G,e)$ in $\eqsch$, where $e$ is the trivial group.
Then the induced map
$$\pi_* : \HWGM{a}X\to \HWM {a} {X/G}$$
is an isomorphism for all $a\in \bZ$. In particular, if $X$ is projective smooth over $k$, we have
$\HWM {a} {X/G}=0$ for $a\not=0$.
\end{theo}

Note that the last statement is non-trivial since $X/G$ may be singular even though $X$ is smooth.

\medskip

For $k=\ol k$ and $M$ uniquely divisible Theorem~\ref{MP.intro} follows from the yoga of weights. 
To see this let us for simplicity
assume $k=\mathbb C$. In terms of the weight filtration on singular cohomology
with compact support, see \cite{D2}, one has
\[
H^W_a(X,G;M) = \Hom_\bQ ( W_0 H^a_c(X_{\rm an} , \bQ )^G ,M ) ,
\]
where the upper index $G$ stands for $G$-invariants. It is not
difficult to check that 
\[
H^a_c( ( X/G)_{\rm an} , \bQ ) \stackrel{\simeq}{\to}  H^a_c(X_{\rm an} , \bQ )^G 
\]
is an isomorphism, from which Theorem~\ref{MP.intro} follows in this
special case.

\medskip

Our basic strategy to show Theorem \ref{MP.intro} in general is to
introduce another (arithmetic)   homology theory
\[
(X,G) \mapsto \KHGM {a} X\quad (a\in \bZ_{\geq 0})
\]
called equivariant Kato homology with coefficient $M$ as a
replacement for the weight zero part of singular cohomology with
compact support above. Here  we assume that $M$ is a torsion $\Lam$-module.
As in case of the weight homology, we simply write $\KHGM {a} X=\KHM {a} X$ in case $G=e$.
For this homology theory McKay principle is rather easily shown: 
 
\begin{prop}\label{MPKH.intro}(Proposition \ref{eqKH.prop1})
For $(X,G)\in \eqsch$, $\pi: X\to X/G$ induces an isomorphism
$$
\pi_* : \KHGM{a}X\isom \KHM {a} {X/G}\;\;\text{for all } a\in \bZ.
$$ 
\end{prop}

Theorem \ref{MP.intro} follows from Proposition \ref{MPKH.intro} and the following.

\begin{theo}\label{KHW.intro}(Theorem \ref{mainthm1})
There exists a map of homology theories on $\eqsch$:
\[
\KHGM * X \to \HWGM * X
\]
which is an isomorphism if $k$ is a purely inseparable extension of a finitely generated field and $M=\Linf:=\bQ/\Lam$.
\end{theo}

The proof of Theorem \ref{KHW.intro} relies on Theorem \ref{KeS.intro} below 
on the cohomological Hasse principle.
The cohomological Hasse principle originally formulated as conjectures by K.Kato \cite{K}, concerns a certain complex of Bloch-Ogus type on an arithmetic scheme $X$ (which means a scheme of 
finite type over $\bZ$):
\begin{multline}\label{eq.KC1}
\cdots\rmapo{\partial} \sumd X a H^{a+1}(x,\nz(a))\rmapo{\partial} 
\sumd X {a-1} H^{a}(x,\nz(a-1))\rmapo{\partial}\cdots\\
\cdots\rmapo{\partial} \sumd X 1 H^{2}(x,\nz(1))\rmapo{\partial} \sumd X 0 H^{1}(x,\nz)
\end{multline}
Here $\Xd a$ denotes the set of points $x\in X$ such that $\dim\overline{\{x\}}=a$
with the closure $\overline{\{x\}}$ of $x$ in $X$, and the term in degree $a$ is 
the direct sum of the Galois cohomology $H^{a+1}(x,\nz(a))$ of the residue fields
$\k(x)$ for $x\in \Xd a$ (for the coefficients $\nz(a)$, see \cite[Lemma 1.5]{KeS}. 
If $(n,\ch(\k(x)))=1$, $\nz(a)$ is the Tate twist of $n$-roots of unity).
The homology groups of the above complex is denoted by $\KHnz a X$ and called
the Kato homology of $X$ (in case $X$ has a component flat over $\bZ$, one need modify the definition 
in order to take into account contributions of $\bR$-valued points of $X$ but we ignore it).
Kato \cite{K} conjectured

\begin{conj}\label{KatoconjG} 
Let $X$ be a proper smooth scheme over a finite field, or a regular scheme proper flat over $\bZ$. Then 
$$
\KHnz a X =0 \qfor a\not=0.
$$
\end{conj}

In case $\dim(X)=1$ the Kato conjecture is equivalent to the Hasse principle for the Brauer group of 
a global field (i.e. a number field or a function field of a curve over a finite field), 
a fundamental fact in number theory.
\medbreak

A generalization of the Kato homology to the (non-equivariant) case over a finitely generated field 
was introduced in \cite{JS1} and the equivariant Kato homology $\KHGM {a} X$ is its equivariant version
(it agrees with $\KHnz a X$ in case $G=e$ and $M=\nz$).
As for non-equivariant case, the following result has been shown.

\begin{theo}\label{KeS.intro}(\cite{KeS})
Assume that $k$ is a purely inseparable extension of a finitely generated field.
For a proper smooth scheme $X$ over $k$ and for a prime to $\ell\not=\ch(k)$, we have 
$$
\KHqzl a X =\indlim n \KHlnz a X =0 \qfor a\not=0.
$$
If $k$ is finite, the same holds by replacing $\qzl$ by $\nz$ with $n$ prime to $\ch(k)$.
\end{theo}

Key ingredients of the proof of Theorem \ref{KeS.intro} is Deligne's theorem \cite{D} 
on the Weil conjecture and Gabber's refinement of de Jong's alteration.
\medbreak

Theorem \ref{KHW.intro} is shown by using Theorem \ref{KeS.intro}.
As a corollary of Theorem \ref{KHW.intro} and Proposition \ref{MPKH.intro}, 
we get the following extension of Theorem \ref{KeS.intro} to a singular case.

\begin{coro}\label{KeS.intro.cor}(Corollary \ref{mainthm1.cor})
Let $k$ be a purely inseparable extension of a finitely generated field.
Let $X$ be a proper smooth scheme with an action of a finite group $G$ over $k$.
For a prime to $\ell\not=\ch(k)$, we have 
$$
\KHqzl a {X/G} =0 \qfor a\not=0.
$$
If $k$ is finite, the same holds by replacing $\qzl$ by $\nz$ with $n$ prime to $\ch(k)$.
\end{coro}
\bigskip

The paper is organized as follows. In \S\ref{eqWhom} we explain the
properties of the equivariant weight homology theory and state McKay principle for the
equivariant weight homology and its corollaries. 
The construction of the equivariant weight homology theory occupies four section
\S\ref{envelopes} through \S\ref{descentii}. The sections
\S\ref{eqcohc} and \S\ref{eqethom} are preliminaries for the
construction of the equivariant Kato homology and the proof of its
basic properties given in \S\ref{eqKhom}. In \S\ref{proof} a
comparison theorem of the equivariant Kato homology and the weight
homology is proved. McKay principle for the equivariant weight
homology stated in \S\ref{eqWhom} is deduced from the comparison theorem. 
In \S\ref{csfundgroup} we relate the fundamental group of the dual complex to the classifying group of cs-coverings.
Using this and the McKay principle stated in \S\ref{eqWhom}, we prove the main theorem \ref{thm.MPhtintro} in \S\ref{MacKayhomtype}.

\smallskip

Recently, A.\ Thuillier observed that one  can deduce the main results  of our paper using
the theory of Berkovich spaces as developed in his article \cite{Th}. This Berkovich approach has the
advantage that it does not use resolution of singularities.

\medbreak

A part of this work was done while the second author's stay at University of Regensburg.
He thanks Uwe Jannsen and the first author cordially for hospitality and financial support.
We thank Sam Payne for explaining us his work on configuration complexes and Yoshihiko Mitsumatsu for helpful discussions on algebraic topology.

\section{Equivariant weight homology}\label{eqWhom}

\noindent
By $k$ we denote a  field and let $\Lam $ be either $\bZ$ or $\bZ[1/p]$, where
$p$ is the exponential characteristic of $k$. 
Let $\eqsch$ be the category of pairs $(X,G)$ of a finite group $G$ and a $G$-scheme $X$ which is quasi-projective over $k$
(i.e. $X\in \sch$ equipped with a left action of $G$ over $k$). 
For objects $(X,G)$ and $(Y,H)$ of $\eqsch$, the morphisms $(X,G)\to (Y,H)$ are
 pairs $(\phi,f)$ of a group homomorphism $\phi:G\to H$ and a map of 
$G$-schemes $f:X\to   Y$, where we endow $Y$ with the induced $G$-action.
Note that the category $\cC_{\eq/k}$ has fibre products. 
We will often write just $X$ for an object $(X,G)\in \eqsch$ when it is clear from the context 
that $X$ is endowed with a group-action.
\medbreak

We say $(\phi,f)$ is strict if $G=H$ and $\phi$ is the identity.
For fixed $G$, let $\schG\subset \eqsch$ be the subcategory of $G$-schemes and strict morphisms. A morphism in $\schG$ is simply denoted by $f:X\to Y$.
Notice that in case $G=e$ the trivial group, $\schG$ is identified with $\cC_{/k}$, i.e.\ the category of quasi-projective schemes over $k$.
\medbreak

A strict open immersion (resp.\ strict closed immersion, resp.\ strict proper morphism) in $\eqsch$ means 
an open immersion (resp.\ closed immersion, resp.\ proper morphism) in $\schG$ for some $G$.
An equivariant simplicial scheme is a simplicial object in $\schG$ for some $G$.

\begin{defn}\label{def.eqhom}
Let $\eqschp$ be the category with the same objects as $\eqsch$ and with morphisms in $\eqsch$ whose
underlying morphisms of schemes are proper. Let $\Lambda$ be a commutative ring.
An equivariant homology theory $H=\{H_a\}_{a\ge 0}$ with values in $\Lam$-modules on $\eqsch$ is
a sequence of covariant functors:
$$
H_a(-): \eqschp  \rightarrow \ModL
$$
satisfying the following conditions:

\medskip

\begin{itemize}
\item [$(i)$] For each strict open immersion $j:V\hookrightarrow X$ in $\eqsch$, 
there is a map $j^*:H_a(X)\rightarrow H_a(V)$,
associated to $j$ in a functorial way.
\item [$(ii)$] If $i:Y\hookrightarrow X$ is a strict closed immersion in $\eqsch$,
with open complement $j:V\hookrightarrow X$, there is a long exact sequence
(called localization sequence)
$$
\cdots\rmapo{\partial} H_a(Y) \rmapo{i_*} H_a(X) \rmapo{j^*} H_a(V) \rmapo{\partial} H_{a-1}(Y)
\longrightarrow \cdots.
$$
(The maps $\partial$ are called the connecting morphisms.) 
\item [$(iii)$]
The sequence in $(ii)$ is functorial with respect to proper morphisms or strict open
immersions in $\eqsch$, in the following sense.
Consider a commutative diagram in $\eqsch$
\[
\begin{CD}
(U',G') @>{j'}>> (X',G') @<{i'}<< (Z',G') \\
@VV{f_U}V @VV{f}V @VV{f_Z}V \\
(U,G) @>{j}>> (X ,G) @<{i}<<  (Z,G), \\
\end{CD}
\]
such that the squares of the underlying schemes are cartesian, and 
$i$ (resp. $i'$) is a strict closed immersion and $j$ (resp. $j'$) is its open complement.

If $G=G'$ and $f$ is a strict open immersion, the following diagram is commutative.
$$
\xymatrix@C=10pt{
 H_a(Z,G) \ar[r] \ar[d]_{f_Z^*} &   H_a(X,G) \ar[r] \ar[d]_{f^*}  & H_a(U,G) \ar[r] \ar[d]_{f_U^*}  & 
H_{a+1}(Z,G)  \ar[d]_{f_Z^*}  \\
  H_a(Z',G) \ar[r]  &   H_a(X',G) \ar[r]  &  H_a(U',G) \ar[r]  &  H_{a+1}(Z',G) 
}
$$

If $f$ is proper, the following diagram is commutative.
$$
\xymatrix@C=10pt{
 H_a(Z',G') \ar[r] \ar[d]_{(f_Z)_*} &   H_a(X',G') \ar[r] \ar[d]_{f_*}  & H_a(U',G') \ar[r] \ar[d]_{(f_U)_*}  & 
H_{a+1}(Z',G')  \ar[d]_{(f_Z)_*}  \\
  H_a(Z,G) \ar[r]  &   H_a(X,G) \ar[r]  &  H_a(U,G) \ar[r]  &  H_{a+1}(Z,G) 
}
$$
\end{itemize}

\medskip

A morphism between homology theories $H$ and $H'$ is a
morphism $\phi: H \rightarrow H'$ of functors on $\eqschp$, which is compatible with the long exact sequences from (ii).
\end{defn}
\medbreak

\begin{defn}\label{def.cS}
We call $(X,G) \in \eqsch$ primitive if $G$ acts transitively on the set of irreducible
components of $X$. Let $\cSeq\subset \eqsch$ (resp. $\cSG\subset \schG$) be the full subcategory of objects 
whose underlying schemes are smooth projective over $k$ and $\cS^{\rm prim}_{\eq/k}  \subset \cSeq$ (resp. $\cSGpr\subset \cSG$) be the full subcategories of primitive objects.
\end{defn}

\medskip

Let $p$ be the exponential characteristic of $k$. 
Let $\tk$ be the perfection of $k$. 
We consider the following condition on resolution of singularities:
\begin{enumerate}
\item[{$\eqRS$}] :
For $(X,G)\in \eqscht$ with $X$ reduced, there are $(X',G)\in \eqscht$ with
$X'$ smooth over $\tk$ and a strict, projective, birational morphism $f:X' \to X$ such that
$f$ is an isomorphism over the regular locus $X_{reg}$ of $X$ and that
the reduced part of $f^{-1}(X-X_{reg})$ is a simple normal crossing divisor on $X'$.
\end{enumerate}

\begin{rem}\label{eqRS.rem1}
\begin{itemize}
\item[(1)]
The condition $\eqRS$ includes as a special case $G=e$ (non-equivariant) resolution of singularities 
proved by Hironaka \cite{H} assuming $\ch(k)=0$.
Under the same assumption $\eqRS$ is a consequence of canonical resolution of singularities 
shown in \cite{BM}.
\item[(2)]
In case $\ch(k)>0$, $\eqRS$ would be false in general if we did not
work over the perfection $\tk$ of  $k$.
\end{itemize}
\end{rem}

\medbreak

\begin{theo}\label{thm.Whom}(see Construction ~\ref{constweighthom})
Assume the following conditions:
\begin{enumerate}
\item[$(\bigstar)$]
$\displaystyle{\Lam=\bZ[\frac{1}{p}]}$, or $\Lam=\bZ$ and $\eqRS$ holds.
\end{enumerate}
For a $\Lam$-module $M$, there exists a homology theory on $\eqsch$ with values in $\Lam$-modules:
\[
(X,G) \mapsto \HWGM {a} X
\]
called equivariant weight homology. It satisfies:
\begin{itemize}
\item[$(i)$]
there are canonical isomorphisms for $(X,G)\in \cSeq$ 
$$\HWGM a X \cong \left\{ \begin{array}{lr}
M^{\Xcd 0/G} & \text{ for } a=0 \\
0 & \text{ for } a\ne 0 \end{array} \right. $$ 
compatible with pushforward in $\cSeq$, where 
$\Xcd 0$ is the set of the generic points of $X$.
\item[$(ii)$]
$M\mapsto \HWM {*} -$ is a covariant functor from $\Lam$-modules to the category of equivariant homology theories,
\item[$(iii)$] for a short exact sequence of $\Lam$-modules
\[
0\to M_1 \to M_2 \to M_3 \to 0
\]
there is a natural long exact sequence 
\[
\cdots \to H^W_a(-;M_1) \to H^W_a(-;M_2) \to H^W_a(-;M_3) \to H^W_{a-1}(-;M_1)\to \cdots ,
\]
\item[$(iv)$]
for $(X,G)\in \cC_{\rm eq/k}$ there is a natural convergent spectral sequence 
\[
E^2_{a,b} = \mathrm{Tor}^{\Lam}_{a}(M,H^W_{b}(X,G;\Lam) )   \Longrightarrow    \HWGM {a+b} X,
\]
\item[$(v)$]
$\HWGL {a} X$ is a finitely generated $\Lam$-module for all $a\ge 0$,
\item[$(vi)$] 
For extensions of fields $k\subset k'$ there are natural pushforward maps
\[
\tau_{k'/k}: \HWGM {a} {X_{k'}} \to  \HWGM {a} X ,
\]
where the left hand side weight homology is relative to the base field $k'$. It satisfies:
\begin{itemize}
\item[$(a)$]
If $k'$ is purely inseparable over $k$, $\tau_{k'/k}$ is an isomorphism.
\item[$(b)$]
For fixed $(X,G)\in \cC_{\rm eq/k}$ and $a\ge 0$,
and for a directed system of fields $\{k_\alpha\}_{\alpha\in I}$ with $k_\alpha \subset k_{\alpha'}$ for $\alpha' \succeq \alpha$ and $\varinjlim_\alpha k_\alpha = k' $ the inverse system 
$$\{ \HWGM {a} {X_{k_\alpha}}\}_\alpha$$ becomes stationary for big $\alpha $  and equal to 
$ \HWGM a {X_{k'}}$.
\end{itemize}
\end{itemize}
\end{theo}
If $G=e$ is trivial, we write $\HWM a X$ for $\HWGM a X$.
\medbreak

In \S\ref{descenti} and \S\ref{descentii} we give a descent construction of the equivariant weight homology.
In the non-equivariant case this is due to Gillet--Soul\'e
(\cite{GS1}, \cite{GS2}) and Jannsen \cite{J1}, relying on ideas of
Deligne~\cite{SGA4}  Expos\'e $\rm V^{bis}$.
If $\ch(k)>0$ our construction depends on Gabber's refinement of de Jong's alteration theorem \cite{Il}, see Theorem~\ref{Gabberres}. 
\medskip

\begin{defn}\label{def.Gstrict} \mbox{}
A $G$-stable simple normal crossing divisor $E$ on a smooth scheme over $k$ with a $G$-action
is $G$-strict if for any irreducible component $E_i$ of $E$ and for any $g\in G$, $g(E_i)=E_i$ or $E_i\cap g(E_i)=\varnothing$.
\end{defn}

\begin{exam}\label{Whom.ex}
Let $X\in \schG$ be a $G$-strict simple normal crossing divisor (Definition \ref{def.Gstrict}).
Then $\HWGM {*} X$ is the homology of the complex:
$$ 
\cdots\to M^{\pi_0(X^{[a]})/G} \rmapo{\partial} M^{\pi_0(X^{[a-1]})/G} 
\rmapo{\partial} \cdots \rmapo{\partial} M^{\pi_0(X^{[0]})/G}. 
$$
Here $X_1,\dots, X_N$ are the $G$-primitive components of $X$ and
$$
 X^{[a]} =\underset{1\leq i_0< \cdots< i_a\leq N}{\coprod} X_{i_0,\dots,i_a}
\quad (X_{i_0,\dots,i_a} = X_{i_0}\cap\cdots\cap X_{i_a})\;,
$$
and $\pi_0(-)$ denotes the set of connected components, and the differentials $\partial$ are obvious alternating ones.
Note that assuming equivariant resolution of singularities one can essentially calculate $H^W$ by these recipes up to extensions of groups.
\end{exam}

\medskip

One of the main results of this paper is the following. Our proof in Section~\ref{proof} uses arithmetic results
due to Jannsen and the authors
 on Kato's cohomological Hasse principle (\cite{J1}, \cite{JS1} and \cite{KeS}).
Recall that $p$ is the exponential characteristic of $k$.

\begin{theo}[McKay principle for weight homology]\label{thm.Whomgraph}
Assume the condition $(\bigstar)$ of Theorem \ref{thm.Whom}.
For $(X,G)\in \eqsch$ let $\pi: X\to X/G$ be the projection viewed as a morphism 
$(X,G)\to (X/G,e)$ in $\eqsch$.
Then the induced map
$$\pi_* : \HWGM{a}X\to \HWM {a} {X/G}$$
is an isomorphism for all $a\in \bZ$.
\end{theo}

\begin{coro}\label{cor.Whomgraph}
Let the assumption be as in \ref{thm.Whomgraph}.
Assume  $(X,G)\in \cS_{\rm eq/k}^{\rm prim}$ . Then 
\[
\HWM {a} {X/G} =
\left.\left\{\begin{gathered}
 M \\
 0 \\
\end{gathered}\right.\quad
\begin{aligned}
&\text{$a=0$}\\
&\text{$a\not=0$}
\end{aligned}\right.
\]
\end{coro}

Corollary~\ref{cor.Whomgraph} is a direct consequence of Theorem~\ref{thm.Whomgraph} and
Theorem~\ref{thm.Whom}$(i)$. 

\begin{coro}\label{cor.Whomgraph2}
Let the assumption be as in \ref{thm.Whomgraph}.
Let $(X,G)\in \eqsch$ and assume $X$ is smooth over $k$.
Let $S\subset X/G$ be a reduced closed subscheme which contains the singular locus of $X/G$.
Let $f:Y\to X/G$ be a proper birational morphism such that $Y$ is smooth over $k$ and $f$ is an isomorphism outside $S$. 
Assume
\begin{itemize}
\item[$(i)$]
$S$ is proper over $k$ (e.g. $S$ is isolated),
\item[$(ii)$]
There exist a strict open immersion $j:X \to \Xb$ in $\schG$ and an open immersion 
$Y\hookrightarrow \Yb$ with $\Xb$ and $\Yb$ projective and smooth over $k$, and a commutative diagram
\[
\xymatrix@C=20pt{
Y \ar[d]_{f}\ar[r] & \Yb\ar[d]_{\overline{f}} \\
X/G \ar[r] & \Xb/G \\
}
\]
where $\overline{f}$ is proper birational.
\end{itemize}
Then the map 
\[\HWM a {f^{-1}(S)}   \stackrel{\sim}{\to} \HWM a S\]
is an isomorphism for all $a\in \bZ$.
In particular, if $S$ is smooth over $k$, 
$ \HWM a {f^{-1}(S)}=0$ for $a\not=0$.
\end{coro}

\begin{rem}\label{cor.Whomgraph2.rem}
The assumption $(ii)$ always holds if $\ch(k)=0$ thanks to \cite{H} and \cite{BM}.
Assuming that $k$ is perfect, the assumption $(ii)$ is a consequence of $\eqRS$.
\end{rem}

\medbreak

Here we deduce Corollary~\ref{cor.Whomgraph2} from Theorem~\ref{thm.Whomgraph}
We may assume that $(\Xb,G)$ is primitive and $\Yb$ is connected.  
Put $\Sb=S\cup (\Xb/G-X/G) \subset \Xb/G$.
By the localization sequence for weight homology we have a commutative diagram with exact rows
\begin{small}
\[
\xymatrix@C=20pt{
 H^W_{a+1}(\Yb) \ar[d] \ar[r] & H^W_{a+1}(\Yb-f^{-1}(\Sb)) \ar[r] \ar[d]^{\simeq} & 
 H^W_{a}(f^{-1}(\Sb)) \ar[d] \ar[r] &  H^W_{a}(\Yb) \ar[d]  \\
 H^W_{a+1}(\Xb/G) \ar[r]  &  H^W_{a+1}(\Xb/G-\Sb) \ar[r] &   H^W_{a}(\Sb) \ar[r]  &   H^W_{a}(\Xb/G) .
}
\]
\end{small}
In the diagram we suppressed the coefficients $M$ for simplicity of notation.
The second vertical isomorphism holds since $Y-f^{-1}(\Sb) \cong \Xb/G-\Sb$.
By Corollary~\ref{cor.Whomgraph} we get
\begin{equation*}\label{eqWhomeq1}  
H^W_a(\Yb,M) \stackrel{\simeq}{\to} H^W_a(\Xb/G,M) \stackrel{\simeq}{\to}
\left.\left\{\begin{gathered}
 M \\
 0 \\
\end{gathered}\right.\quad
\begin{aligned}
&\text{$a=0$}\\
&\text{$a\not=0$}
\end{aligned}\right.
\end{equation*}
Thus, from the above diagram we deduce an isomorphism
\[H^W_a( {f^{-1}(\Sb)};M)   \stackrel{\sim}{\to} H^W_a(\Sb;M ) \]
By the assumption $(i)$, $\Sb$ is the disjoint sum of $S$ and $(\Xb/G-X/G)$.
Hence the above isomorphism induces the desired isomorphism of Corollary~\ref{cor.Whomgraph2}.
$\square$
\medskip

From the proof of Theorem~\ref{thm.Whomgraph} we also deduce the boundedness of $H^W$ in Section~\ref{proof}.  

\begin{prop}\label{prop.bounded}
For $X\in \schG$ and for a $\Lam$-module $M$ we have under condition $(\bigstar)$ of Theorem \ref{thm.Whom} $$\HWGM {a} X=0$$
if $a\ge \dim(X)+1$. 
\end{prop}
Proposition~\ref{prop.bounded} would follow immediately from Theorem~\ref{thm.Whom} under the assumption of a strong form of equivariant resolution of singularities, which we do not have in positive characteristic at the moment. See \cite[Sec.\ 2.5]{GS1} for an analogous argument. Our proof
of Proposition~\ref{prop.bounded} relies on equivariant Kato homology and is of arithmetic nature
(see its proof in the last part of \S\ref{proof}).

\bigskip

\section{Admissible envelopes}\label{envelopes}

\noindent
Let the notation be as in Section~\ref{eqWhom}.
We define $\Lam$-admissible envelopes of equivariant schemes generalizing an idea of Gillet \cite{G1}. 
Then we state an equivariant alteration theorem due to Gabber and sketch a proof. 
This allows us to construct smooth hyperenvelopes of equivariant simplicial schemes.
Recall that $p$ denotes the exponential characteristic of the base
field $k$ and that $\Lam$ is either $\bZ$ or $\bZ[1/p]$.        

\begin{defi}(\cite[1.4.1]{GS1})
A morphism $f:(X,G) \to (Y, G)$ in $\cC_{G/k}$ is called a $\Lam$-admissible envelope if 
the underlying morphism $X\to Y$ is proper and it satisfies the following condition.

If $\Lam=\bZ$, 
for a primitive $G$-equivariant scheme $(\Spec(L),G)$ where $L$ is a finite product of fields,
the induced map $X(L)\to Y(L)$ on the the sets of $G$-equivariant points of $X$ and $Y$ with values
in $(\Spec(L),G)$ is surjective.

If $\Lam\ne \bZ$, 
for each prime number $\ell\not=p$
and for each $G$-equivariant point $P\in Y(L)$ with values in $(\Spec(L),G)$, 
there is a finite $G$-equivariant extension $(\Spec(L'),G) \to (\Spec(L),G)$ of degree prime to $\ell$
and a $G$-equivariant point $P'\in X(L')$ mapping to $P\in Y(L) \hookrightarrow Y(L')$.
\end{defi}

Clearly the system of $\Lam$-admissible envelopes is closed under compositions and base changes.

For the construction of admissible envelopes we need some sort of resolution of singularities over the base field $k$. For this consider the following
condition:

\begin{quote}
$\text{\bf (G)}_{\ell}$ : For a prime number $\ell$ and $(X,G)\in \eqsch$ with $X$ reduced, 
there are $(X',G)\in \eqsch$ with $X'$ smooth over $k$ and a strict, surjective, projective morphism
$f:X' \to X$ which is finite of degree prime to $\ell$ over any maximal point of $X$.
\end{quote}

\begin{theo}[Gabber]\label{Gabberres}
If $k$ is perfect, $\text{\bf (G)}_{\ell}$ holds whenever $\ell\not=p$.
\end{theo}

In case $G=e$  this result due to Gabber is shown in \cite{Il}.
Gabber communicated the following proof \cite{Ga1} to us in the general case.

\begin{proof}
Without loss of generality $X$ is $G$-primitive.
Let $X_1$ be an irreducible component of $X$ and put 
\[
G_1=\{g\in G |\, g(X_1)=X_1 \}.
\] 
Let $G_\ell$ be an $\ell$-Sylow subgroup of $G_1$ and denote by $D\hookrightarrow Y=X_1/G_\ell$ 
the locus over which $G_\ell$ does not act freely on $X_1$. According to \cite[Theorem 1.3]{Il}, 
there is a dominant generically finite morphism $g:Y' \to Y$ of degree prime to $\ell$
such that $Y'$ is smooth and $D'=g^{-1}(D)$ is a simple normal crossing divisor on $Y'$. 
Let $Y_1$ be the normalization of $Y'$ in $k(Y') \otimes_{k(Y)} k(X_1)$. 
As $Y_1$ is tamely ramified over $Y'$, we use Abhyankar's lemma \cite[XIII Proposition 5.1]{SGA1} in order to find \'etale locally over $Y'$ an embedding of $Y_1$ into a Kummer \'etale extension of the log scheme $(Y',D')$. 
By \cite[ Proposition~6.6(c)]{Il} $(Y_1,D_1)$ is log regular where $D_1$ is the reduced preimage of $D'$. 
Using results of Kato, Niziol and Gabber \cite{K2} and \cite{Ni} on resolution of log regular schemes,
we find an equivariant resolution of singularities $(X'_1,G_\ell)\to (Y_1,G_\ell) $. 
If we let $(X',G)$ be the equivariant smooth scheme given by $X'=G \times_{G_\ell} X'_1$,
the morphism $f:(X',G) \to (X,G)$ is as demanded in $\text{\bf (G)}_{\ell}$.
\end{proof}

\medskip

\begin{prop}\label{env.existence}
Assume the condition $(\bigstar)$ of Theorem \ref{thm.Whom} and further that $k$ is perfect.
Then, for any $(X,G)\in \cC_{G/k}$ with $X$ reduced, there is a $\Lam$-admissible 
envelope $f:(X',G)\to (X,G)$ with $X'$ smooth over $k$.
\end{prop}

\begin{proof}
In case $\Lam=\bZ$ and $\eqRS$ holds, the assertion is proved by the same argument as \cite{GS1} Lemma 2.
We prove the assertion in case $\displaystyle{\Lam=\bZ[\frac{1}{p}]}$.
We use the induction on $d=\dim(X)$. Start with an arbitrary prime $\ell\not=p$.
According to Theorem \ref{Gabberres}, there is $(X'_\ell,G)$ with $X'_\ell$ smooth over $k$ and 
a surjective proper morphism $f_\ell:(X'_\ell,G) \to (X,G)$ which is finite
of degree prime to $\ell$ over each  maximal point of $X$. For every prime $\ell'$ in the set
\[
L=\{ \ell' \text{ prime} \;|\, \ell'\not=p \text{ and } \ell' | \deg (f_\ell)  \},
\]
choose $(X'_{\ell'},G)$ with $X'_{\ell'}$ smooth over $k$ and a surjective proper morphism
$f_{\ell'}:(X'_{\ell'},G) \to (X,G)$ which is finite of degree prime to
$\ell'$ over each maximal point of $X$. Consider the object
\[
X'_{\rm gen} =   X'_\ell \amalg	 \coprod_{\ell'\in L} X'_{\ell'} 
\]
in $\cC_{G/k}$. Then $f_{\rm gen}:(X'_{\rm gen},G) \to (X,G)$ is a $\Lam$-admissible envelope over a dense open $G$-equivariant subscheme $U\subset X$. Now by induction 
there is a  $\Lam$-admissible envelope $f_{\rm sp}:(Y',G) \to (Y,G)$ with $Y'$ smooth over $k$, 
where $Y=X-U$ with reduced structure. The morphism
\[
f=f_{\rm gen} \amalg f_{\rm sp} : X'_d \amalg Y' \to X
\]
is the envelope we are looking for.
\end{proof}

\medskip

In the next section $\Lam$-admissible equivariant hyperenvelopes will play a central role.

\begin{defi}\label{def.hypenv}
A $\Lam$-admissible hyperenvelope is an equivariant morphism of simplicial schemes  $f:X_\bullet \to Y_\bullet$ in $\eqsch^\Delta$ such that 
the morphism
\[
X_a  \to (\mathrm{cosk}^{Y_\bullet}_{a-1}\mathrm{sk}_{a-1} X_\bullet)_a
\]
is a $\Lam$-admissible envelope for any $a\ge 0$.
\end{defi}

The system of $\Lam$-admissible hyperenvelopes of equivariant simplicial schemes is closed under composition and base change. This can be deduced from the characterization of hypercovers given in 
\cite[Lem.\ 2.6]{GS2} and the analogous statement for $\Lam$-admissible envelopes.
\medbreak

Combining Proposition~\ref{env.existence} and Theorem \ref{Gabberres} with the proof of \cite[Lemma 2]{GS1},
we deduce:

\begin{prop}\label{Ladmienvelop}
Assume the condition $(\bigstar)$ of Theorem \ref{thm.Whom} and further that $k$ is perfect.
Then for $X_\bullet\in \cC^\Delta_{G/k} $ there is a 
$\Lam$-admissible hyperenvelope $Y_\bullet \to X_\bullet$ in
$\Sch^\Delta_G$ with $Y_a/k$ smooth for all $a\ge 0$.
\end{prop}

\bigskip

\section{Descent construction of homology theories} \label{exthom}

In this section we explain how to extend equivariant homology functors defined for smooth projective varieties to all varieties, assuming that they satisfy a certain descent property. The construction is an equivariant version of Gillet--Soul\'e \cite[\S2]{GS1}
\medbreak

Let $C_+ (\Mod_\Lam )$ be the category of homological complexes of
$\Lam$-modules vanishing in negative degrees.
By definition an equivariant simplicial scheme $(X_\bullet, G)$ is a simplicial object in $\schG$ for 
a fixed finite group $G$. We say that $(X_\bullet,G)$ has proper face maps if the underlying morphisms
of the face maps are proper.

\begin{defi}\label{def.homfunc}
A homology functor 
\[
\Phi: \eqsch \to  C_+(\Mod_\Lam ) 
\]\
is a covariant functor on $\eqschp$ which is also contravariant with respect to strict open immersions and which satisfies the
following conditions.

\begin{enumerate}
\item[($i$)]
If $i:Y\to X$ is a strict closed immersion in $\eqsch$ with strict open  complement $j:V \to X$ the composition 
\[
\Phi(Y) \rmapo {i_*} \Phi(X) \rmapo{j^*} \Phi(V)
\]
vanishes as a map of complexes.
\item[($ii$)]
In the situation ($i$) the induced homomorphism 
\[
{\rm cone}[  \Phi (Y) \stackrel{i_*}{\to} \Phi (X) ]  \longrightarrow  \Phi(V)
\]
is a quasi-isomorphism.
\item[($iii$)]
Consider a diagram in $\eqsch$
\[
\begin{CD}
(V',G') @>{j'}>> (X',G')  \\
@VV{f_V}V @VV{f}V  \\
(V,G) @>{j}>> (X ,G) , \\
\end{CD}
\]
such that the diagram of the underlying schemes is cartesian, and $f$ and $f_V$ are proper, and
$j$ and $j'$ are strict open immersions.
Then the following square commutes:
\[ 
\xymatrix{
\Phi(X',G')  \ar[r]^{j^*}  \ar[d]_{(f_V)_*}  &  \Phi(V',G') \ar[d]^{f_*} \\
\Phi(X,G)  \ar[r]_{j^*}            &   \Phi(V,G)  
}
\]
\item[($iv$)]
For $(X,G)$ and $(Y,G)$ in $\cSeq$, the natural morphism of complexes
\[
\Phi(X) \oplus \Phi(Y)   \stackrel{\sim}{\to}  \Phi(X \coprod Y)
\]  
is an isomorphism.
\end{enumerate}
We also use the notion of a homology functor on $\schG$ for a fixed finite group $G$, whose definition is analogous. 
Using total complexes one extends a homology functor $\Phi: \eqsch \to C_+(\Mod_\Lam ) $
to a functor
\[
\Phi: \eqschp^\Delta \to C_+(\Mod_\Lam),
\]
where $ \eqschp^\Delta$ is the category of equivariant simplicial
schemes with proper face maps and proper morphisms. 

\end{defi}

\medbreak

\begin{defi}\label{def.homfunc2}
For a homology functor $\Phi$, the functors $H(\Phi)=\{H_a(\Phi)\}_{a\ge 0}$ with
\[
H_a(\Phi): \eqsch \to \ModL\;;\; (X,G) \to H_a(\Phi(X,G))
\]
form a homology theory on $\eqsch$ (this is obvious from the definition). 
We call $H(\Phi)$ the homology theory associated to $\Phi$.
\end{defi}

\medbreak

\begin{defi}\label{def.prehomfunc}
A pre-homology functor 
\[
F: \cSeq \to  C_+(\Mod_\Lam ) 
\]
is a covariant functor on $\cSeq$ which satisfies the following condition:
For $(X,G)$ and $(Y,G)$ in $\cSeq$, the natural morphism of complexes
\[
F(X) \oplus F(Y)   \stackrel{\sim}{\to}  F(X \coprod Y)
\]  
is an isomorphism.

\end{defi}

\begin{defi}\label{ZS} (\cite{GS1})
\noindent
\begin{enumerate}
\item[(1)]
Let $\LSG$ be the category with the same objects as $\cSG$ and with morphisms 
\[
{\Hom}_{\Lam\cSG} (X,Y) =\Lam \Hom_{\cSG} (X,Y),
\]
the free $\Lam$-module on the set ${\rm Hom}_{\cSG} (X,Y)$.
It is easy to check that $\LSG$ is a $\Lam$-additive category and the coproduct
$X\oplus Y$ of $X,Y\in Ob(\cS_{/k})$ is given by $X\coprod Y$. 
\item[(2)]
For a simplicial object in $\cSG$:
$$
X_\bullet\;:\; \quad
\cdots X_2 \;
\begin{matrix}
\rmapo{\delta_0}\\
\lmapo{s_0}\\
\rmapo{\delta_1}\\
\lmapo{s_1}\\
\rmapo{\delta_2}\\
\end{matrix}
\; X_1 \;
\begin{matrix}
\rmapo{\delta_0}\\
\lmapo{s_0}\\
\rmapo{\delta_1}\\
\end{matrix}
\; X_0 ,
$$
we define a homological complex in $\LSG$:
$$
\Lam X_\bullet\;:\; \quad
\cdots \; \to X_n \rmapo{\partial_n} X_{n-1} \to \cdots,\quad
\partial_n=\sum_{j=0}^n (-1)^j \delta_j.
$$
This gives a functor
\[
 \cSG^\Delta \to C_+(\LSG),
\]
where $C_+(\LSG)$ is the category of homological complexes in $\LSG$.
\end{enumerate}
\end{defi}

Let $F :\cSG \to C_+(\Mod_\Lam )$ be a pre-homology functor.
Using total complexes, one extends $F$ to a functor 
\begin{equation}\label{estprehomfunc}
F : C_+(\LSG) \to C_+(\Mod_\Lam ).
\end{equation}
For a simplicial object $X_\bullet$ in $\cSG$, we write
$F(X_\bullet)=F(\Lam X_\bullet)$.

\medskip

Consider the following descent conditions for a homology functor $\Phi$ on $\eqsch$
(resp.\ a pre-homology functor $F$ on $\cSeq$). 
For a homology functor $\Phi$, we let $\Phi_{|\cSeq} $ denote the pre-homology functor obtained
by restricting $\Phi$ to the subcategory $\cSeq$ of $\eqsch$.

\begin{quote}
  $\text{\bf (D)}_{\Phi}$ : For any finite group $G$ and for any 
  $\Lam$-admissible equivariant hyperenvelope of simplicial schemes
  $f:X_\bullet \to Y_\bullet$ in $\cC^\Delta_{G/k *}$, the map
\[
f_* : \Phi(X_\bullet ) \to \Phi(Y_\bullet )
\]
is a quasi-isomorphism.
\end{quote}
\medskip

\begin{quote}
$\text{\bf (D)}_{F}$
For any finite group $G$ and for any $X_* \in C_+(\LSG)$ with the
property that $\Phi_{|\cSeq}(X_*) $ is acyclic for any homology functor
$\Phi$ satisfying 
$\text{\bf (D)}_{\Phi}$, 
we have that $F(X_*) $ is acyclic. 
\end{quote}

\begin{lem}\label{PHIF.lem}
\begin{itemize}
\item[(1)]
If a homology functor $\Phi$ satisfies $\text{\bf (D)}_{\Phi}$, then $F=\Phi_{|\cSeq} $ satisfies 
$\text{\bf (D)}_{F}$.
\item[(2)]
Let $f:X_\bullet\to Y_\bullet$ be a $\Lam$-admissible hyperenvelope in ${\schG}$ such that
$X_\bullet$ and $Y_\bullet$ are objects in ${\cSG}$.
If a pre-homology functor $F$ satisfies $\text{\bf (D)}_{F}$, then
$f_*: F(X_\bullet)\to F(Y_\bullet)$ is quasi-isomorphism.
\end{itemize}
\end{lem}
\begin{proof}
(1) is obvious. We prove (2).
Put $C_*={\rm cone}(\Lam X_\bullet \rmapo{\Lam f} \Lam Y_\bullet)$ in $C_+(\LSG)$.
For a homology functor $\Phi$, we have
\[
\Phi_{|\cSeq}(C_*)={\rm cone}(\Phi(X_\bullet) \rmapo{f_*} \Phi(Y_\bullet)).
\]
Hence, if $\Phi$ satisfies $\text{\bf (D)}_{\Phi}$, $\Phi_{|\cSeq}(C_*)$ is acyclic since $f_*$ is a quasi-isomorphism. 
Thus $\text{\bf (D)}_{F}$ implies
\[
F(C_*)={\rm cone}(F(X_\bullet) \rmapo{f_*} F(Y_\bullet))
\]
is acyclic, which proves the desired conclusion of Lemma \ref{PHIF.lem}.
\end{proof}

\medskip

The heuristic idea for the following construction is that there should
be a universal extension up to quasi-isomorphisms of a pre-homology
functor $F$ satisfying 
$\text{\bf (D)}_{F}$ to something similar to a homology functor $\Phi^F$ satisfying $\text{\bf (D)}_{\Phi^F}$. However we do not know how to make this precise, but we can at least construct the expected homology theory associated to $\Phi^F$. The construction we explain below is essentially due to Gillet-Soul\'e \cite{GS1}.
 
\begin{const}\label{funconst} 
We now assume the following:
\begin{itemize}
\item[$(i)$]
$k$ is perfect,
\item[$(ii)$]
the condition $(\bigstar)$ of Theorem \ref{thm.Whom} holds.
\end{itemize}
Given a pre-homology functor $F$ with values in $\Lam$-modules satisfying $\text{\bf (D)}_{F}$,
we construct an equivariant homology theory $H^F$ (see Definition~\ref{def.eqhom}) 
with the following properties:
\begin{enumerate}
\item 
There is a canonical isomorphism of homology groups 
$$
H^F_a|_{\cSeq} \cong H_a( F ) \;\;\;\; \text{ for all } a\in \Z
$$ 
compatible with equivariant pushforwards.
\item 
For $X_\bullet \in \cSG^\Delta$ and $X\in \schG$ proper and for a $\Lam$-admissible equivariant hyperenvelope $X_\bullet \to X$, there is a natural descent spectral sequence
\[
E^1_{a,b}  =  H_{b}^F(X_a,G)  \;\; \Longrightarrow  \;\;  H^F_{a+b}(X,G).
\]
\item 
If $F=\Phi_{|\cSeq}$ for a homology functor $\Phi$ satisfying $\text{\bf (D)}_{\Phi}$,
we have a canonical isomorphism of homology theories $H^F \cong H(\Phi)$.
\item 
Let $F \to F'$ be a morphism of pre-homology functors satisfying $\text{\bf (D)}_{F}$ and $\text{\bf (D)}_{F'}$ with associated homology theories $H^F$ and $H^{F'}$. 
Then there is a canonical morphism of homology theories $H^F\to H^{F'}$. 
Furthermore, if for every $(X,G)\in \cSeq$, the map $F(X,G) \to F'(X,G)$ is a quasi-isomorphism,
the induced morphism of homology theories $ H^F\to H^{F'}$ is an isomorphism.
\item 
Let $F''$ be ${\rm cone}(F \to F'$) for a morphism of pre-homology functors $F\to F'$ satisfying $\text{\bf (D)}_{F}$ and $\text{\bf (D)}_{F'}$ and let $H^F,H^{F'}$ and $H^{F''}$ be the associated homology theories. Then for $(X,G)\in \cSeq$ there is a long exact sequence
\[
\cdots \to H^F_{a}(X,G) \to H^{F'}_{a}(X,G) \to H^{F''}_{a}(X,G) \to H^F_{a-1}(X,G) \to \cdots
\]
compatible with equivariant proper pushforward and strict equivariant open pullback.
\item 
Assume that for every $(X,G)\in \cSeq$ and every $a\in \Z$, the complex $F(X,G)$ consists of 
finitely generated $\Lam$-modules. Then the homology groups 
$H^{F}_{a}(X,G) $ are finitely generated $\Lam$-modules for every $(X,G)\in \eqsch$.
\item 
Let $M$ be a $\Lam$-module. Assume that for every $(X,G)\in \cSeq$, the complex $F(X,G)$ consists of 
flat $\Lam$-modules. There is a natural spectral sequence
\[
E^2_{a,b} = \mathrm{Tor}^{R}_{a}(M,H_{b}^{F}(X,G) )   \Longrightarrow    H_{a+b}^{F\otimes_{\Lam} M } (X,G),
\]
where $F\otimes_{\Lam} M(X,G)=F(X,G)\otimes_{\Lam} M$.
\end{enumerate}
\end{const}

\medskip

Now we describe how this constructions is accomplished. The non-equivariant case is explained in detail by Gillet-Soul\'e \cite[\S2]{GS1} and the equivariant case works similarly, which we explain briefly. 
\medbreak

We start with a pre-homology functor $F$ satisfying $\text{\bf (D)}_{F}$.
For $(X,G)\in \eqsch$ choose an equivariant compactification $j:(X,G) \to (\ol X, G)$, where
$\Xb$ is projective over $k$ and $j$ is a strict open immersion with dense image.
It exists by a construction explained in the first part of Section~\ref{eqcohc}. 
Let $i:(Y,G)\hookrightarrow (\ol X, G)$ be the strict closed immersion of the (reduced) complement of $X$ in $\ol X$. By Proposition \ref{Ladmienvelop} we can choose a commutative diagram in $\schG^\Delta$
\begin{equation}\label{exthom.sqcom-1}
\xymatrix{
(Y_\bullet,G)  \ar[r]^{i^\Delta} \ar[d]^{f} &   (\ol X_\bullet,G) \ar[d]^{g}\\
(Y,G)  \ar[r]^{i}  & (\ol X,G) .
}
\end{equation}
where $(Y_\bullet,G),\; (\Xb_\bullet,G) \in \cSG^\Delta$ and $f,g$ are
$\Lam$-admissible equivariant hyperenvelopes. Then we set
\[
H^F_{a}(X,G)  = H_a ( {\cone}[ F(Y_\bullet) \rmapo{i^\Delta_*} F(\ol X_\bullet ) ] ).
\]
It can be shown that the homology groups $H^F_{a}(X,G) $ are independent of the choices made, 
i.e.\ are unique up to canonical isomorphism, and satisfy the axioms of an equivariant homology theory.
\medbreak

We sketch a proof of independence. Let us first fix the compactification $(X,G) \to (\ol X,G)$. 
Choose a different commutative diagram in $\schG^\Delta$
\begin{equation}\label{exthom.sqcom-2}
\xymatrix{
(Y'_\bullet,G)  \ar[r]^{i^{'\Delta}} \ar[d]^{f'} &   (\ol X'_\bullet,G) \ar[d]^{g'}\\
(Y,G)  \ar[r]^{i}  & (\ol X,G) .
}
\end{equation}
where $(Y'_\bullet,G),\; (\Xb'_\bullet,G) \in \cSG^\Delta$ and $f',g'$ are
$\Lam$-admissible equivariant hyperenvelopes. 
Thanks to Lemma \ref{PHIF.lem}(2), the method of \cite[\S2.2]{GS1} produces a canonical isomorphism
\[
{\cone}[ F(Y_\bullet) \to F(\ol X_\bullet ) ]  \cong   {\cone}[ F(Y'_\bullet) \to F(\ol X'_\bullet ) ]
\;\;\text{in } D_+(\Mod_\Lam ).
\]
\medbreak

As for independence from compactifications, consider another compactification $(X,G) \to (\ol X',G)$ with complement $(Y',G)$. By a standard trick, see \cite[\S2.3]{GS1},
one can assume that there is a commutative diagram
\begin{equation}\label{exthom.sqcom}
\xymatrix{
(Y',G)  \ar[r] \ar[d]_{\pi} &   (\ol X',G) \ar[d]^{\pi} \\
(Y,G)  \ar[r]  & (\ol X,G) .
}
\end{equation}
Choose smooth $\Lam$-admissible hyperenvelopes mapping to \eqref{exthom.sqcom} 
\begin{equation}\label{exthom.sqcom2}
\xymatrix{
(Y'_\bullet,G)  \ar[r] \ar[d]_{\pi_\bullet} &   (\ol X'_\bullet,G) \ar[d]^{\pi_\bullet} \\
(Y_\bullet,G)  \ar[r]  & (\ol X_\bullet,G) .
}
\end{equation}
where all terms are in $\cSG^\Delta$.
We are ought to show that the induced map
\[
{\cone} [ F(Y'_\bullet) \to F(\ol X'_\bullet ) ]   \stackrel{\pi_*}{\to} 
{\cone} [ F(Y_\bullet) \to F(\ol X_\bullet ) ] 
\]
is a quasi-isomorphism. This follows from the following.

\begin{claim}
Let $\Xi$ denote the total complex of the diagram in $C_+(\Lam \cSG )$ obtained by
applying the functor in Definition \ref{ZS}(2) to the diagram \eqref{exthom.sqcom2}.
Then 
\[
 F(\Xi)  \simeq  0  \;\;\;\; \text{ in  } D_+(\Mod_\Lam ).
\]
\end{claim}

For any homology functor $\Phi$ satisfying $\text{\bf (D)}_{\Phi}$ we have
\[
\Phi_{|\cSeq}(\Xi) \simeq 
\Tot
\begin{pmatrix}
\Phi(Y'_\bullet)  &\to&  \Phi(\ol X'_\bullet)  \\
\downarrow & &\downarrow \\
\Phi(Y_\bullet)  &\to& \Phi(\ol X_\bullet)\\
\end{pmatrix}
\simeq  \Tot
\begin{pmatrix}
\Phi(Y')  &\to&  \Phi(\ol X')  \\
\downarrow & &\downarrow \\
\Phi(Y)  &\to& \Phi(\ol X)\\
\end{pmatrix}
\simeq 0 \;\; \text{ in  } D_+(\Mod_\Lam ),
\]
where the first quasi-isomorphism is obvious from the definition, the second follows from
$\text{\bf (D)}_{\Phi}$, and the last follows from Definition \ref{def.homfunc}$(ii)$.
By $\text{\bf (D)}_{F}$, this implies the claim.
$\square$

\bigskip

\section{Descent criterion for homology functors}\label{descenti}

In this section and the next we will explain two basic descent theorems which give criteria 
for a homology functor $\Phi$ (resp.\ a pre-homology functor $F$) to satisfy the descent condition
$\text{\bf (D)}_{\Phi}$ (resp. $\text{\bf (D)}_{F}$) from \S\ref{exthom}. 
This will then allow us to apply the methods of Section~\ref{exthom} in order to construct equivariant weight homology theory.
\medbreak
 
We say that an equivariant homology functor (Definition \ref{def.homfunc})
$$\Phi:  \eqsch \to  C_+({\Mod}_\Lam ) $$
has {\em quasi-finite flat pullback} if the following conditions hold.

\begin{itemize}
\item 
For a strict quasi-finite flat morphism $f:(X,G)\to (Y,G)$ there is a functorial pullback $f^*: \Phi(Y,G) \to \Phi(X,G)$ which coincides with the usual open pullback if $f$ is an open immersion.
\item 
Consider a commutative diagram in $\schG$
\[
\begin{CD}
(V',G) @>{j'}>> (X',G)  \\
@VV{f_V}V @VV{f}V  \\
(V,G) @>{j}>> (X ,G) , \\
\end{CD}
\]
such that the diagram of the underlying schemes is cartesian, and $f$ and $f_V$ are proper, and
$j$ and $j'$ are quasi-finite flat. Then the following diagram commutes:
\[ 
\xymatrix{
\Phi(X')  \ar[r]^{j^*}  \ar[d]_{f_*}  &  \Phi(V') \ar[d]^{f_*} \\
\Phi(X)  \ar[r]_{j^*}            &   \Phi(V)  .
}
\]
\item 
For a strict finite flat morphism $f:(X,G) \to (Y,G)$ of degree $d$, the composition
\[
f_* \circ f^* : \Phi(Y,G)  \to \Phi(Y,G) 
\] 
is multiplication by $d$.
\end{itemize}
 
\medskip

Our first fundamental descent theorem reads.

\begin{theo}\label{descentthmi}
Consider an equivariant homology functor $\Phi$ on $\schG$, which we
assume to have quasi-finite flat pullback if $\Lam\ne \bZ $.
Then $\Phi$ satisfies $\text{\bf (D)}_{\Phi}$.
\end{theo}
\begin{proof}
  The proof of Theorem \ref{descentthmi} is practically the same as
  that of \cite[Theorem~3.4 and Theorem~3.9] {GS2} (see also
  \cite[\S5.6]{GS2}). One just has to add a $G$-action everywhere. The
  details are left to the readers. 
\end{proof}

\bigskip

\section{Descent criterion for pre-homology functors and weight homology}\label{descentii}

We explain a descent criterion for equivariant pre-homology functors (Definition \ref{def.prehomfunc})
in terms of equivariant Chow motives. We start with explaining the construction of the latter.

\medskip

By the naive equivariant Chow group of $(X,G)\in \eqsch$, one usually means the Chow group of the
quotient variety $\CH_a(X/G) $ ($a\ge 0$). The problem with this definition is that there is no intersection product and no flat pullback on naive Chow groups.
Edidin, Graham and Totaro (\cite{EG}, \cite{To}) suggested a definition of refined Chow groups. Choose a $k$-linear representation $G\times V \to V$ such that for all $1\ne g\in G$ 
the codimension in $V$ of
\[
\Fix (g,V) = \{ v\in V | g\cdot v = v \text{ and } g|_{k(v)}=id \}
\]
is greater than the dimension of $X$. Here we view $V$ as an equivariant affine space. Now define 
\[
\CH_a(X,G) = \CH_{a+d_V}(X\times V /G)\;\; (d_V=\dim(V)),
\]
where we endow $X \times V$ with the diagonal $G$-action. In \cite{EG} it is shown that this is independent of
choices of $V$. For $(X,G)$ primitive of dimension $d$, we write
\[
\CH^{a}(X,G) = \CH_{d-a}(X,G).
\]

\medskip

In what follows we consider covariant Chow motives with coefficients
in $\Lam$. We do not care about pseudo-abelian envelopes nor
non-effective motives. Fix a finite groups $G$. The objects of our
category of Chow motives, denoted by $\ChowG$, are the objects of
$\cSG$. For $(X,G)\in \cSG$ we write $\Mo (X,G)$ for the induced
object in $\ChowG$. For $X,Y\in \cSG$, we let the morphisms in the category of Chow motives be
\[
\Hom_{\ChowG}(\Mo (X,G),\Mo (Y,G)) = 
\underset{i\in I}{\bigoplus} \;\CH_{d_i}(X_i \times Y ,G) \otimes_\Z \Lam .
\]
where $X=\underset{i\in I}{\coprod}  X_i$ with $(X_i,G)\in \cSG$ primitive of dimension $d_i$.
Compositions are defined as compositions of correspondences. Clearly
$\Mo$ defines a covariant functor from $\cSG$ to $\ChowG$ by associating to a morphism
$f: X\to Y$ its graph in $X\times Y$. It extends in an obvious way to a $\Lam$-linear functor  
\begin{equation}\label{eq.Mo}
\Mo : \LSG \to \ChowG ,
\end{equation}
where $\LSG$ is defined in Definition \ref{ZS}(1).
Clearly the functor $\Mo$ in \eqref{eq.Mo} induces a canonical $\Lam$-linear functor  
\begin{equation}\label{eq.Mo2}
\Mo :  C_+(\Lam\cSG)\to C_+(\ChowG)
\end{equation}
which induces a canonical $\Lam$-linear functor  
\begin{equation}\label{eq.Mo3}
\Mo :  K_+(\Lam\cSG)\to K_+(\ChowG),
\end{equation}
where $K_+(\Lam\cSG)$ (resp. $K_+(\ChowG)$) is the homotopy category of $C_+(\Lam\cSG)$ (resp. $C_+(\ChowG)$).
\medbreak

\medbreak

We can now state our fundamental descent theorem for pre-homology functors.

\begin{theo}\label{DescentiiThm}
Given a pre-homology functor 
$$F: \cSeq \to  C_+(Mod_\Lam ),$$
assume that for any finite group $G$, there is an additive $\Lam$-linear functor 
$$\tilde{F}_G: \ChowG  \to  C_+(Mod_\Lam )$$
such that the restriction $F_{|\cSG}$ of $F$ to $\cSG$ is the composition of 
$\tilde{F}_G$ with $\cSG\to \LSG\rmapo{\Mo} \ChowG $.
Then $F$ satisfies the descent property $\text{\bf (D)}_{F}$.
\end{theo}
 
 \smallskip
\begin{const}\label{constweighthom}(Construction of weight homology)
 Fix a $\Lam$-module $M$. For $(X,G)\in \cSeq$ we set 
\begin{equation}\label{examweighthom.eq1}
F^W_M(X,G) = \Hom_\Z (\CH^{0}(X,G) , M )[0] \simeq M^{\Xcd 0/G}[0]\;\in C_+(\ModL),
\end{equation}
where $\Xcd 0$ is the set of the generic points of $X$.
It is easy to see that $F^W_M$ is a pre-homology functor such that $(F^W_M)_{|\cSG}$ extends to a $\Lam$-linear functor on $\ChowG$.
Therefore Theorem~\ref{DescentiiThm} shows that $F^W_M$ satisfies the descent property $\text{\bf (D)}_{F^W_M}$.
\medbreak

We now assume that the condition $(\bigstar)$ of Theorem \ref{thm.Whom} holds.
If one assumes further that $k$ is perfect, we can apply Construction~\ref{funconst} to $F^W_M$ to get
a homology theory on $\eqsch$:
$$(X, G) \mapsto \{ \HWGM a X \}_{a\geq 0} . $$
If $k$ is not perfect field $k$, we define
\begin{equation}\label{examweighthom.eq2}
\HWGM {*} X = \HWGM {*} {X\otimes_k{\tk}},
\end{equation}
where $\tk$ is the perfection of $k$.
This homology theory is called {\em weight homology with coefficient $M$}.
\medbreak

It is clear from Construction~\ref{funconst} that Theorem~\ref{thm.Whom}$(i)$ through $(v)$ hold.
To define the map $\tau_{k'/k}$ in $(vi)$, we argue as follows. 
By \eqref{examweighthom.eq2}, we may replace $k$ and $k'$ by their perfections to assume that those fields 
are perfect. Let $\iota:\cSeq \to \cSeqp$ be the base change functor $X\to X\otimes_k k'$.
Write $F$ and $F'$ for the pre-homology functors \eqref{examweighthom.eq1} on $\cSeq$ and $\cSeqp$ respectively.
Then we have a morphism $F'\circ \iota \to F$ of pre-homology functors on $\cSG$, induced by natural maps
\[
\CH^{0}(X,G) \to \CH^{0}(X\otimes_k k',G)\qfor (X,G)\in \cSeq.
\]
By Construction \ref{funconst}(4), this induces a morphism of corresponding homology theories
\[
H^{F'\circ\iota}=H^{F'}\circ \iota \to H^{F}
\]
which gives the desired map $\tau_{k'/k}$. Theorem \ref{thm.Whom}$(vi)(a)$ is obvious 
and $(b)$ is easily shown by using $(v)$.
\end{const}
\medskip
 
To prepare the proof of Theorem\ \ref{DescentiiThm} we have to introduce a rigidified version of the equivariant Gersten complex for algebraic $K$-theory \cite{Q}.
In fact one could also use the Gersten complex for Milnor $K$-theory.
Fix a sequence of linear $G$-representations $V_j$ ($j\ge 0$) over $k$, $d_{V_j}= \dim_k V_j$, with linear surjective equivariant transition maps
$V_j\to V_{j-1}$ such that for any $g\ne 1 \in G$, the codimension of $\Fix(g,V_j)$ in $V_j$ converges to infinity as $j \to \infty$.
For $X\in \eqsch$ define
\[
R_{q,i}(X,G) = \lim_{\longrightarrow \atop j} \bigoplus_{x\in (X\times V_j)/G \atop \dim \overline{\{ x \} }=q+i+d_{V_j}} K_i(k(x)) \otimes_\Z  \Lam, 
\]
where the limit is taken over flat pullbacks.
For fixed $q\in \Z$ we get a homological complex of $\Lam$-modules:
\[
R_q(X,G):\; \cdots \rmapo{\partial} R_{q,i}(X,G) \rmapo{\partial} R_{q,i-1}(X,G) \rmapo{\partial} 
\cdots\rmapo{\partial} R_{q,0}(X,G),
\]
where $R_{q,i}(X,G)$ sits in degree $i$ and $\partial$ are boundary maps arising from localization theory for algebraic $K$-theory \cite{Q}. These complexes form homology functors 
\[
R_q: \eqsch \to C_+(Mod_\Lam )
\]
which have quasi-finite flat pullbacks \cite{Q}. 
Hence we deduce the following proposition from Theorem~\ref{descentthmi}. 

\begin{prop}\label{descentR}
For any $q\in \Z$ the descent property  $\text{\bf (D)}_{\Phi}$ is satisfied for $\Phi=R_q$.
\end{prop}
\bigskip

Recall that the prehomology functor $(R_q)_{|\cSeq}$ defined as the restriction of $R_q$ to $\cSeq$ extends to a canonical functor (cf. \eqref{estprehomfunc}) 
\[
 C_+(\Lam\cSG) \to C_+(\Mod_\Lam) .
\]
which we denote by $R_q$ for simplicity.
With almost verbatim the same proof as for Theorem~1 of \cite{GS1} we obtain the following
(cf. \eqref{eq.Mo3}):

\begin{prop}\label{propdescentGS}
For $X_* \in K_+(\Lam\cSG)$ with $R_q(X_*)$ acyclic for all $q\ge 0$ we have
\[
\Mo (X_*) =0 \in K_+(\ChowG).
\]
\end{prop}
\bigskip\noindent
{\it Proof of Theorem \ref{DescentiiThm}}:
Consider $X_* \in C_+( \Lam \cSG ) $ such that $\Phi_{|\cSeq}(X_*)$ is acyclic for every homology functor $\Phi$ satisfying $\text{\bf (D)}_{\Phi}$. 
In particular, by Proposition~\ref{descentR} $R_q(X_*)$ is acyclic for every $q\in \Z$.
By Proposition~\ref{propdescentGS} this implies that $\Mo(X_*)=0$ in  $K_+( \ChowG )$. 
Since $F$ factors through $\ChowG$ by assumption, we conclude that $F(X_*)=0$ in $K_+( \Mod_\Lam )$.
$\square$

\bigskip

\section{Equivariant cohomology with compact support}\label{eqcohc}

\noindent
Let the notation be as in \S\ref{eqWhom}. 
For $X\in \schG$, we consider  $G$-sheaves $\cF$ of torsion $\Lam$-module on $X_{\et}$, which means that $G$ acts in compatible way  with the action of $G$
on $X$. The $G$-sheaves of torsion $\Lam$-modules on $(X,G)$ form an abelian category $\ShGXL$. 
For a morphism $f:X\to Y$ in $\schG$, the direct image functor $f_*$ on 
sheaves extends to the functor
\begin{equation}\label{eqhom.-3}
f_*: \ShGL X \to \ShGL Y.
\end{equation}
For a morphism $(\phi,f):(X,G)\to (Y,H)$ in $\eqsch$, the pullback functor $f^*$
on sheaves extends to the functor
\begin{equation}\label{eqhom.-2}
 (\phi,f)^*:{\it Shv}_{H}( Y) \to \ShGL X.
\end{equation}
\medbreak

For $\cF\in \ShGXL$, we define its equivariant cohomology groups with compact support.
First we choose an equivariant compactification
$$
j:X\hookrightarrow \Xb
$$
where $j$ is an open immersion in $\schG$ and $\Xb$ is proper over $k$.
Such $j$ always exists: Take a (not necessarily equivariant) compactification
$j':X\hookrightarrow Z$ and consider the map
$$
\tau: X\hookrightarrow \overbrace{Z\times_k \cdots \times_k Z}^{|G|\; times}
$$
induced by $j'\cdot g$ for all $g\in G$. Then one takes $\Xb$ to be the closure of 
the image of $\tau$ and $j$ to be the induced immersion. 
Global sections on $\ol X$ form a functor
$$
\Gamma(\ol X, -): \Shv_G (\ol X ) \to \ModLG\;;\; \cF\to \Gamma (\ol X,\cF).
$$
with its derived functor 
$$
R\Gamma (\ol X, -): D^+( \Shv_G ( \ol X ) ) \to D^+(\ModLG).
$$
Let
$$
\Gamma(G, - ): \ModLG \to \ModL\;;\; M\to M^G
$$
be the functor of taking the $G$-invariants of $\Lam[G]$-modules and
$$
R\Gamma(G, - ): D^+(\ModLG) \to D^+(\ModL)
$$
be its derived functor. 
\medbreak

We define equivariant cohomology groups 
by
\[
H^n (X,G;\cF)=H^n(R\Gamma(G,R\Gamma (X, \cF)))\qfor n\in \bZ
\]
and we define equivariant cohomology groups with compact support as
\begin{equation}\label{eqhom.eq1}
\HcGX n \cF=H^n(R\Gamma(G,R\Gamma (\ol X,j_! \cF)))\qfor n\in \bZ .
\end{equation}
We have a convergent spectral sequence
\begin{equation}\label{eqhom.eq1.5} 
E_2^{a,b}= H^a(G,H^b_c(X,\cF)) \Rightarrow \HcGX {a+b} \cF.
\end{equation}

\begin{lem}\label{eqhom.lem0}
Let $(X,G)\in \eqsch$ with $G$ acting trivially on $X$.
Then there is an isomorphism of $\Lam$-modules
\[
H^n(X,G; \cF ) = H^n(X, B_G( \cF ) ).
\]
Here $B_G$ denotes the cohomological Bar resolution functor
\[
B_G( - ) : \Shv_G( X )  \to C^+ (\Shv (X) ).
\]
\end{lem}

\bigskip

Let $i:Z\to X$ be a closed immersion in $\schG$ and $j:U\to X$ be its open complement. For $\cF\in \ShGL X$ we have an exact sequence in $\ShGL X$ 
$$
0\to j_!j^*\cF \to \cF \to i_*i^*\cF\to 0
$$ 
which induces a long exact sequence
\begin{equation}\label{eqhom.eq1.6}
\cdots \rmapo{\delta} \HcGU n {j^*\cF} \rmapo{j_*} \HcGX n \cF \rmapo{i^*} \HcGZ n{i^*\cF} 
\rmapo{\delta} \HcGU {n+1}{j^*\cF} \to \cdots
\end{equation}
\bigskip

Take $f:(X,G) \to (Y,H)$ in $\eqsch$ and choose a commutative diagram in $\eqsch$ 
\begin{equation}\label{eqhom.eq2}
\begin{CD}
(X,G) @>{j_X}>> (\ol X,G)\\
@VV{f}V @VV{\fb}V  \\
(Y,H) @>{j_Y}>> (\ol Y,H)\\
\end{CD}
\end{equation}
where $j_X$ and $j_Y$ are strict open immersions, and $\Xb$ and $\Yb$ are proper over $k$.
Such a choice is always possible by the above construction. 
\medbreak

Assume $G=H$, $f$ is strict and the underlying morphism $f:X\to Y$ is flat and quasi-finite.
If $p$ is invertible in $\Lam$ 
we construct the pushforward map
\begin{equation}\label{eqhom.eq1.8}
f_*: H^{n}_c(X,G; f^*\cF) \to H^{n}_c(Y,G; \cF).
\end{equation}
By SGA XVIII Theorem 2.9, we have the trace morphism
$Rf_!f^*\cF \to \cF$. By applying $(j_Y)_!$, it induces a map
$R\fb_* {j_X}_! f^*\cF \to {j_Y}_!\cF$, which gives rise to a $G$-equivariant map
\[
R\Gamma(\Xb,{j_X}_!f^*\cF)\to R\Gamma(\Yb,{j_Y}_!\cF) .
\]
This induces the map \eqref{eqhom.eq1.8} by applying $R\Gamma(G,-)$
and taking cohomology.
\bigskip

Assume the underlying morphism $f:X\to Y$ is proper.
Then we construct the pullback map
\begin{equation}\label{eqhom.eq1.7}
f^*: H^n_c(Y,H; \cF) \to \HcGX n {f^*\cF}.
\end{equation}
By the properness of $f$ the map
$X\to Y\times_{\Yb} \Xb$ induced from \eqref{eqhom.eq2} is an isomorphism
and we have the base change isomorphism
$\fb^*{j_Y}_!\cF \isom {j_X}_! f^*\cF$ by SGA2 XVII 5.2.1. 
It gives rise to an equivariant map
\[
f^*: R\Gamma(\Yb,{j_Y}_!\cF) \to R\Gamma(\Xb,{j_X}_!f^*\cF)
\]
which induces 
\[
R\Gamma(H,R\Gamma(\Yb,{j_Y}_!\cF)) \to R\Gamma(G,R\Gamma(\Yb,{j_Y}_!\cF)) \rmapo{f^*}
R\Gamma(G,R\Gamma(\Xb,{j_X}_!f^*\cF)),
\]
where the first map is induced by $G\to H$. This induces the map \eqref{eqhom.eq1.7}.
If $f$ is a strict closed immersion, it is easy to see that $f^*$ coincides with the pullback map
in \eqref{eqhom.eq1.6}.

\medbreak

For an extension $k'/k$ of fields and for $(X,G)\in \eqsch$, put 
\[
X'=X\times_{\Spec(k)}\Spec(k') \in \schGkp \text{ with $f:X'\to X$ the natural map,}
\]
where the action of $G$ on $X'$ is induced from that on $X$ via the base change.
We construct the base change map
\begin{equation}\label{eqhom.eq2.7}
\iota_{k'/k}: \HcGX n \cF \to H_c^n(X',G;f^*\cF).
\end{equation}
For a $G$-equivariant compactification $j_X:X\hookrightarrow \Xb$,
we have a cartesian diagram  
\begin{equation*}\label{eqhom.eq2.9}
\begin{CD}
X' @>{j_{X'}}>> \Xb' \\
@VV{f}V  @VV{\fb}V  \\
X @>{j_X}>> \Xb\\
\end{CD}
\end{equation*}
where $\Xb'=\Xb\times_{\Spec(k)}\Spec(k')$ with $\fb:\Xb'\to \Xb$ the natural map.
It gives rise to the base change isomorphism
$\fb^*{j_X}_! \cF \isom {j'_X}_! f^*\cF$, which induces the map \eqref{eqhom.eq2.7} 
by the same argument as before.

\medskip

\begin{lem}\label{eqhom.lem1}
Let the notation be as above.
\begin{itemize}
\item[(1)]
Consider a commutative diagram in $\eqsch$
\[
\begin{CD}
(U',G') @>{j'}>> (X',G') @<{i'}<< (Z',G') \\
@VV{f_U}V @VV{f}V @VV{f_Z}V \\
(U,G) @>{j}>> (X ,G) @<{i}<<  (Z,G) 
\end{CD}
\]
such that the squares of the underlying schemes are cartesian, and 
$i$ (resp. $i'$) is a strict closed immersion and $j$ (resp. $j'$) is its open complement.
Take $\cF\in \ShGXL$ and put $\cF' = f^* \cF$.
\medbreak

If $f$ is proper, the following diagram is commutative.
\begin{small}
$$
\xymatrix@C=10pt{
 \HcG n(U,G;j^*\cF) \ar[r] \ar[d]_{f^*} &   \HcG n(X,G;\cF) \ar[r] \ar[d]_{f^*}  & \HcG n(Z,G;i^*\cF) \ar[r] \ar[d]_{f^*}  & \HcG {n+1}(U,G;j^*\cF)  \ar[d]_{f^*}  \\
 \HcG n(U',G';{j'}^*\cF') \ar[r]  &   \HcG n(X',G';\cF') \ar[r]  &   \HcG n(Z',G';{i'}^*\cF') \ar[r]  &  \HcG {n+1}(U',G';{j'}^*\cF') 
}
$$
\end{small}
\medbreak

If $G=G'$ and $f$ is strict, flat and quasi-finite and 
if $\cF\in \ShGYL$ is the limit of the subsheaves annihilated by integers prime to $\ch(k)$, the following diagram is commutative.
\begin{small}
$$
\xymatrix@C=10pt{
 \HcG n(U',G;{j'}^*\cF') \ar[r] \ar[d]_{f_*} &   \HcG n(X',G;\cF') \ar[r] \ar[d]_{f_*}  & \HcG n(Z',G;{i'}^*\cF') \ar[r] \ar[d]_{f_*}   & \HcG {n+1}(U',G;{j'}^*\cF')  \ar[d]_{f_*}  \\
 \HcG n(U,G;{j}^*\cF) \ar[r]  &   \HcG n(X,G;\cF) \ar[r]  &   \HcG n(Z,G;{i}^*\cF) \ar[r]  &  \HcG {n+1}(U,G;{j}^*\cF) 
}
$$
\end{small}
Assuming further that $f$ is finite flat of degree $d$, the composition $f_*f^*$ is the multiplication 
by $d$.
\item[(2)]
Consider morphisms in $\schG$
\[ 
(U,G) \rmapo{j} (X ,G) \lmapo{i} (Z,G)
\] 
where $i$ is a strict closed immersion and $j$ is its open complement.
Take $\cF\in \ShGXL$.
Let $k'/k$ be an extension of fields and $X',U',Z',\cF'$ be the base changes of $X,U,Z,\cF$ via $k'/k$
respectively. Then the following diagram is commutative.
\begin{small}
$$
\xymatrix@C=10pt{
 \HcG n(U,G;j^*\cF) \ar[r] \ar[d]_{\iota_{k'/k}} &   \HcG n(X,G;\cF) \ar[r] \ar[d]_{\iota_{k'/k}}  & \HcG n(Z,G;i^*\cF) \ar[r] \ar[d]_{\iota_{k'/k}}  & \HcG {n+1}(U,G;j^*\cF)  \ar[d]_{\iota_{k'/k}}  \\
 \HcG n(U',G;{j'}^*\cF') \ar[r]  &   \HcG n(X',G;\cF') \ar[r]  &   \HcG n(Z',G;{i'}^*\cF') \ar[r]  &  \HcG {n+1}(U',G;{j'}^*\cF') 
}
$$
\end{small}
\end{itemize}
\end{lem}
\begin{proof}\label{eqcohc.prop} 
The corresponding fact in the non-equivariant case (i.e. the case $G$ is trivial)
are well-known and the lemma is proved by the same argument as in that case. 
\end{proof}
\bigskip

For $(X,G)\in \eqsch$, $X$ is quasi-projective over $k$ by definition, and by SGA1, V\S1, 
the geometric quotient $X/G$ exists. Let $\pi: X\to X/G$ be the natural projection,
which is viewed as a map $(X,G)\to (X/G,e)$ in $\eqsch$, where $e$ is the trivial group.
A sheaf $\cF$ of torsion $\Lam$-modules on $X/G$ gives rise to $\pi^*\cF\in \ShGXL$ and a map
\begin{equation}
\pi^*: H^n_c(X/G,\cF) \to \HcGX n {\pi^*\cF}.
\end{equation}
by the formalism of \eqref{eqhom.eq1.7}.

\begin{prop}\label{eqhom.pr1}
If $G$ acts freely on $X$, $\pi^*$ is an isomorphism for all $n$.
\end{prop}
\begin{proof}
Take an equivariant compactification $j: X \hookrightarrow \Xb$  and let $\pib :\Xb\to \Xb/G$ be the projection to the geometric quotient.
Then $j$ induces a commutative diagram
$$
\begin{CD}
X @>{j}>> \Xb\\
@VV{\pi}V @VV{\pib}V  \\
X/G @>{j}>> \Xb/G\\
\end{CD}
$$

As in \eqref{eqhom.-3} we have a canonical pushforward functor
\[
\pi_* : \ShGL X \to \ShGL {X/G},
\]
where we consider $X/G$ with the trivial $G$-action,
and similarly for $\ol \pi$.
Putting $Y=X/G$ and $\Yb=\Xb/G$.
we have
$$
\begin{aligned}
H^n_c( \Xb , G ; \pi^*\cF ) &= H^n( \Xb,G ; j_!\pi^*\cF ) \\
&= H^n (\Yb,G ;  \pib_*j_!\pi^*\cF) \\
&= H^n (\Yb,G ;  j_!\pi_*\pi^*\cF) \\
&=  H^n (\Yb, B_G ( j_!\pi_*\pi^*\cF) ) \\
&= H^n (\Yb, j_! B_G ( \pi_*\pi^*\cF) ) \\
&= H^n_c(\Yb ,  B_G ( \pi_*\pi^*\cF) ) 
\end{aligned}
$$
where the second  and third equality hold since $\pi$ and $\pib$ are finite. 
The fourth equality holds by Lemma~\ref{eqhom.lem0}.
Finally it is well known that there is a quasi-isomorphism  $B_G ( \pi_*\pi^*\cF ) \simeq
\cF$, since $\pi$ is an \'etale covering with Galois group $G$. Indeed, on stalks $ \pi_*\pi^*\cF$
is given by $G$-representations which are induced from the trivial group.  This completes
the proof of the proposition.
\end{proof}
\bigskip

A similar isomorphism to that in Proposition \ref{eqhom.pr1} holds in a radicial situation:

\begin{lem}\label{eqhom.radicial}
\begin{itemize}
\item[(1)]
Let $f: (X,G) \to (Y,G)$ be a map in $\schG$ where the underlying morphism is finite, surjective and radicial. For $\cF \in \ShGL Y$, the pullback map \eqref{eqhom.eq1.7}
\[
f^* : \HcGY n \cF \to \HcGX n {f^* \cF}
\]
is an isomorphism.
\item[(2)]
For a purely inseparable extension $k'/k$ of fields and for $(X,G)\in \eqsch$, the base change map \eqref{eqhom.eq2.7}
\[
\iota_{k'/k}: \HcGX n \cF \to H_c^n(X',G;f^*\cF)
\]
is an isomorphism. Here $X'=X  \otimes_k k'$.
\end{itemize}
\end{lem}
\begin{proof}
This follows from \cite{SGA4} XVIII 1.2 and XVII 5.2.6.
\end{proof}

\bigskip

\section{Equivariant \'etale homology}\label{eqethom}

Let the assumption be as in \S\ref{eqcohc}.
Let $G_k$ be the absolute Galois group of $k$ and $M$ be a discrete $\Lam[G_k]$-module which is torsion as 
a $\Lam$-module and continuous as a $G_k$ module. 
We let $M$ denote the corresponding \'etale sheaf on $\Spec(k)$.
When $M$ is finite, we write $M^\vee = \Hom_{\Lam} (M,\Linfty)$, the dual $\Lam[G_k]$-module,
where $\Linfty$ be the torsion  $\Lam$-module $\Q / \Lam$.

\begin{defn}\label{eqhom.def1}
The equivariant \'etale homology groups of $(X,G)\in \eqsch$ with coefficient in a finite 
$\Lam[G_k]$-module $M$ are defined as
$$
H^{\et}_{a}(X,G;M)= \HcGX a {M^\vee}^\vee  \quad\text{ for } a\geq 0,
$$
where $M^\vee$ on the right hand side denotes the pullback to $X$ of the \'etale sheaf ${M^\vee}$
on $\Spec(k)$. For a $\Lam[G_k]$-module $M$ which is torsion as a $\Lam$-module, we write $$M=\varinjlim_n M_n $$
with $M_n$ finite $\Lam[G_k]$-modules and set
\[
H^{\et}_{a}(X,G;M) = \varinjlim_n H^{\et}_{a}(X,G;M_n).
\]
\end{defn}

By Lemma \ref{eqhom.lem1}, the functor
$$
\eqsch \to \ModL\;;\; (X,G) \to H^\et_a(X,G;M)  \quad (a\in\bZ)
$$
provides us with a homology theory on $\eqsch$ in the sense of Definition \ref{def.eqhom}.
\medbreak

For $(X,G) \in \eqsch$ and an integer $a\geq 0$, let $X_{(a)}$ denotes the set of such 
$x\in X$ that $\dim(\overline{\{x\}})=a$. The group $G$ acts on the set $X_{(a)}$ and 
we let $X_{(a)}/G$ denotes the quotient set. For $x\in X_{(a)}/G$, let $Gx$ be the 
corresponding $G$-orbit in $X$. 
\medbreak

For a homology theory $H$ on $\eqsch$ (Definition \ref{def.eqhom}),  
we have a spectral sequence of homological type,
called the niveau spectral sequence:
\begin{equation}\label{spectralsequence1}
E^1_{a,b}(X,G;M)=\bigoplus_{x\in X_{(a)}/G }   H_{a+b}(x,G;M)~~\Rightarrow ~~H_{a+b}(X,G;M),
\end{equation}
where
$$ 
H_a(x,G;M)=\indlim{V\subseteq \overline{\{Gx\}}} H_a(V,G;M),
$$
with the limit taken over all non-empty $G$-stable open subschemes $V$ in 
the closure $\overline{\{Gx\}}$ of $Gx$ in $X$. 
This spectral sequence was constructed by Bloch--Ogus \cite{BO} in the non-equivariant case.
The same construction works for the equivariant case and one can show the following:

\begin{prop}\label{eqhom.prop0}
For the spectral sequence \eqref{spectralsequence1}, the following facts hold.
\begin{itemize}
\item[(1)]
The spectral sequence is covariant with respect to morphisms in $\eqsch$ whose underlying morphism of schemes
are proper, and contravariant with respect to strict open immersions in $\eqsch$.
\item[(2)]
Consider a commutative diagram in $\eqsch$
\[
\begin{CD}
(V',G') @>{j'}>> (X',G')  \\
@VV{f_V}V @VV{f}V  \\
(V,G) @>{j}>> (X ,G) , \\
\end{CD}
\]
such that the diagram of the underlying schemes is cartesian, and $f$ and $f_V$ are proper, and
$j$ and $j'$ are strict open immersions.
Then the following square commutes:
\[ 
\xymatrix{
(E^1_{a,b}(X',G';M ),d^1_{a,b}) \ar[r]^{(j')^*}  \ar[d]_{(f_V)_*}  &  (E^1_{a,b}(V',G';M ),d^1_{a,b}) \ar[d]^{f_*} \\
(E^1_{a,b}(X,G;M ),d^1_{a,b}) \ar[r]_{j^*}            &   (E^1_{a,b}(V,G;M ),d^1_{a,b})
}
\]
\end{itemize}
\end{prop}

For the homology theory in Definition \ref{eqhom.def1}, the above spectral sequence 
for $(X,G)\in \eqsch$ is written as
\begin{equation}\label{spectralsequence2}
E^1_{a,b}  (  X,G ; M )  =\sumdG X a H^\et_{a+b}(x,G;M )~~\Rightarrow ~~
H^\et_{a+b}(X,G;M ),
\end{equation}

\begin{prop}\label{eqhom.prop1}
For the spectral sequence \eqref{spectralsequence2}, the following facts hold.
\begin{itemize}
\item[(1)]
If $\ch(k)$ is invertible in $\Lam$, the spectral sequence is contravariantly functorial for strict flat quasi-finite maps. For a commutative diagram in $\eqsch$
\[
\begin{CD}
(V',G') @>{j'}>> (X',G')  \\
@VV{f_V}V @VV{f}V  \\
(V,G) @>{j}>> (X ,G) , \\
\end{CD}
\]
such that the diagram of the underlying schemes is cartesian, and $f$ and $f_V$ are proper, and
$j$ and $j'$ are strict and flat quasi-finite, the following square commutes:
\[ 
\xymatrix{
(E^1_{a,b}(X',G';M ),d^1_{a,b}) \ar[r]^{(j')^*}  \ar[d]_{(f_V)_*}  &  (E^1_{a,b}(V',G';M ),d^1_{a,b}) \ar[d]^{f_*} \\
(E^1_{a,b}(X,G;M ),d^1_{a,b}) \ar[r]_{j^*}            &   (E^1_{a,b}(V,G;M ),d^1_{a,b})
}
\]
Assuming further that $f$ is finite flat of degree $d$, the composition $f_*f^*$ is the multiplication 
by $d$.
\item[(2)]
If $f: (X',G) \to (X,G)$ is a map in $\schG$ where the underlying morphism is finite, surjective and radicial, it induces an isomorphism of spectral sequences
\[
(E^1_{a,b}(X',G;M ) ,d^1_{a,b}) \isom (E^1_{a,b}(X,G;M ) ,d^1_{a,b}).
\]
\item[(3)]
For a purely inseparable extension $k'/k$ of fields and for $(X,G)\in \eqsch$, 
there is a natural isomorphism of spectral sequences
\[
(E^1_{a,b}(X',G;M ) ,d^1_{a,b}) \isom (E^1_{a,b}(X,G;M ) ,d^1_{a,b}),
\]
where $X'=X\times_{\Spec(k)}\Spec(k')$ with the $G$-action induced from that on $X$.
\item[(4)]
We have $\EGXM 1 a b=0$ for $b<0$.
The projection $\pi:X\to X/G$ induces isomorphisms for all $a\in \bZ$
$$
\begin{aligned}
\EGXM 1 a 0&\simeq \Big(\sumd X a H^\et_{a}(x;M)\Big)_G\\
&\simeq \sumdy {(X/G)} a H_{a}^\et(y;M)= E^1_{a,0}(X/G;M)
\end{aligned}
$$
and hence an isomorphism of complexes
\[
(\EGXM 1 \bullet 0,d^1_{\bullet,0}) \isom (E^1_{\bullet,0}(X/G;M),d^1_{\bullet,0}).
\]
\end{itemize}
\end{prop}
\begin{proof}
(1) through (3) follow immediately from Lemma \ref{eqhom.lem1} and Lemma \ref{eqhom.radicial}.
To show (4), we may assume $M$ is finite.
In the proof all homology (resp.\ cohomology) functors have coefficients
$M$ (resp.\ $M^\vee$). For simplicity we suppress $M$ in the notation. Recall for $x\in X_{(a)}/G$
$$
\HG {a+b} x =
\indlim{U\subseteq
  \overline{\{Gx\}}}\Hom_\Lambda\big(H_c^{a+b}(U,G),\Lambda_\infty \big),
$$
where the limit is taken over all non-empty $G$-stable affine open subschemes $U$ in 
the closure $\overline{\{Gx\}}$ of $Gx$ in $X$. 
By the affine Lefschetz theorem we have
$H^i_c(U)=0$ for all $i<a$. Hence the spectral sequence \eqref{eqhom.eq1.5} implies
\begin{equation}\label{eqhom.eq3}
\HcG i(U,G)=0\;\text{  for $i<a$},\quad
\HcG a(U,G)\simeq H^a_c(U)^G.
\end{equation}
This immediately implies part (1) of the proposition and an isomorphism
$$
\EGX 1 a 0= \bigoplus_{x\in X_{(a)}/G }   H^\et_{a}(x,G) \simeq \Big(\sumd X a H^\et_{a}(x)\Big)_G
$$
noting that for $x\in X_{(a)}$ we have 
$$
H^\et_{a}(x)=\indlim{V\subseteq \overline{\{x\}}}
\Hom_\Lambda\big(H_c^{a}(V),\Lambda_\infty \big)
$$
where the limit is taken over all non-empty affine open subscheme $V$ in 
the closure $\overline{\{x\}}$ of $x$ in $X$. 
To complete the proof, it remains to show 
\begin{equation}\label{eqhom.eq3.1}
\HG a x \simeq H_a^\et(y)\quad\text{ for 
$x\in X_{(a)}/G$ and $y=\pi(x)$},
\end{equation}
noting $X_{(a)}/G\simeq (X/G)_{(a)}$, an isomorphism of sets. Let $H$
be the image of $G$ in $\Aut(Gx)$, the scheme theoretic automorphism
group, and put $H_x=\{\sigma\in H\;|\;\sigma(x)=x\}$. One can choose a
$H_x$-stable affine open subschemes $U$ in $\overline{\{x\}}$ such
that $H_x$ acts freely on $U$ and $\sigma(U)\cap \sigma'(U)=\varnothing$
for $\sigma,\sigma'\in H$ with $\sigma^{-1}\sigma'\not\in H_x$. Put
$$
G\cdot U=\underset{\sigma\in G}{\bigcup}\sigma(U)=\underset{\sigma\in H/H_x}{\coprod}\sigma(U).
$$
By \eqref{eqhom.eq3}, \eqref{eqhom.eq1.5} implies
$$
\HcG a(U,G)\simeq \big(\underset{\sigma\in H/H_x}{\bigoplus}
H^a_c(\sigma(U))\big)^H = H^a_c(U)^{H_x}\simeq \HcGx a(U,H_x),
$$
where the last group denotes the equivariant cohomology with compact support for 
$(U,H_x)\in \eqsch$. By Proposition \ref{eqhom.pr1} we get
$$
\HcHx a(U,H_x)\simeq H^a_c(U/H_x)=H^a_c(G\cdot U/G),
$$
where the last equality holds since $U/H_x=G\cdot U/G$. Thus we get
\begin{equation}\label{eqhom.eq3.2}
\HG a x \simeq \indlim{U\subseteq \overline{\{x\}}}
\Hom_\Lambda\big(H^a_c(G\cdot U/G),\Lambda_\infty \big).
\end{equation}
As $U$ runs over all $H_x$-stable non-empty open subsets of $\overline{\{x\}}$, $G\cdot U/G$ 
runs over all non-empty open subsets of $T/G$ where $T=\overline{\{Gx\}}$.
Consider the finite morphism $\pi_T:T/G\to S$ where $S\subset X/G$ is the closure of 
$y=\pi(x)$ in $X/G$. 

\begin{claim}
Letting $z\in T/G$ be the generic point (note that $T/G$ is integral),
$\k(z)/\k(y)$ is purely inseparable. In particular $\pi_T$ is an isomorphism on 
the underlying topological spaces.
\end{claim}
\medbreak
The claim follows from \cite[Ch.V \S2.2 Thm.2(ii)]{Bo}.
\medbreak

Thanks to the above claim, we get from \eqref{eqhom.eq3.2} and 
Proposition~\ref{eqhom.radicial}
$$
\HG
 a x \simeq \indlim{V\subseteq \overline{\{y\}}}
\Hom_\Lambda\big(H^a_c(V),\Lambda_\infty \big) = H_a^\et (y),
$$ 
where $V$ ranges over all non-empty open subsets of $\overline{\{y\}}$, 
which proves the desired claim \eqref{eqhom.eq3.1}.
This completes the proof of Proposition \ref{eqhom.prop1}.

\end{proof}

\bigskip

\section{Equivariant Kato homology}\label{eqKhom}

Let the assumption be as in \S\ref{eqethom}.

\begin{defn}\label{def.eqKC}
For $(X,G) \in \eqsch$, the equivariant Kato complex $\KCGM X$
of $X$ with coefficient $M$ is defined as
$$
 E^1_{d,0}(X,G;M) \cdots \to E^1_{a,0}(X,G;M) \to E^1_{a-1,0}(X,G;M) \to \cdots\to E^1_{0,0}(X,G;M ) \;,
$$
where $d=\dim(X)$ and the boundary maps are $d^1$-differentials and $E^1_{a,0}(X,G;M)$ is put in degree $a$.
\end{defn}

In case $G=e$, $\KCGM X$ is simply denoted by $\KCM X$. 
This complex has been introduced first in \cite{JS1} as a generalization of a seminal work
by Kato \cite{K}. The complex $KC(X; \mathbb Z / n \mathbb Z)$ is isomorphic to the complex \eqref{eq.KC1} by
duality, see \cite[Sec.\ 1]{KeS}.

\begin{lem}\label{eqhom.lem2}
The correspondence
$$
\KCGM - :\eqsch \to C_+(\ModL)\;;\;  (X,G)\to \KCGM X
$$
is a homological functor on $\eqsch$ in the sense of Definition~\ref{def.homfunc}. 
\end{lem}
\begin{proof}
The properties \ref{def.homfunc}$(i)$ and $(ii)$ follow from the exact sequence of complexes
$$
0\to \KCGM Y \rmapo {i_*} \KCGM X \rmapo {j^*}  \KCGM U \to 0
$$
for a closed immersion $i:Y\to X$ in $\schG$ and its complement $j:U=X-Y\to X$,
which is an easy consequence of the fact $\Xd a =\Yd a\coprod \Ud a$.
The functoriality for proper morphisms and strict open immersions follows from 
Proposition \ref{eqhom.prop0}(1).
The property \ref{def.homfunc}$(iii)$ follows from Proposition \ref{eqhom.prop0}(2).
The property \ref{def.homfunc}$(iv)$ is obvious.
\end{proof}

\begin{prop}\label{eqKH.prop1}
For $(G,X)\in \eqsch$ with the projection $\pi:X\to X/G$, the natural map
$$
\pi_*: \KCGM X  \to \KCM {X/G}
$$
is an isomorphism of complexes.
\end{prop}
\begin{proof}
This follows from Proposition \ref{eqhom.prop1}(4).
\end{proof}
\medbreak

\begin{prop}\label{eqKH.prop2}
\begin{itemize}
\item[(1)]
If $f: (X',G) \to (X,G)$ is a map in $\schG$ where the underlying morphism is finite, surjective and radicial, it induces an isomorphism of complexes
\[
f_* : \KCGM {X'} \isom \KCGM X.
\]
\item[(2)]
For a purely inseparable extension $k'/k$ of fields and for $(X,G)\in \eqsch$, 
there is a natural isomorphism of complexes
\[
\KCGM {X'} \isom \KCGM X.
\]
where $\Xb'=\Xb\times_{\Spec(k)}\Spec(k')$ with the $G$-action induced from that on $X$.
\end{itemize}
\end{prop}
\begin{proof}
This follows from Proposition \ref{eqhom.prop1}(2) and (3).
\end{proof}
\medbreak

\begin{prop}\label{thm.KHdescent}
 The homological functor $\KCGM -$ on $\eqsch$ satisfies the descent condition 
$\text{\bf (D)}_{KC}$ from Section~\ref{exthom}.
\end{prop}
\begin{proof}
This follows from Proposition \ref{eqhom.prop1}(1) and Theorem \ref{descentthmi}.
\end{proof}

\begin{defn}\label{def.eqKH}
The equivariant Kato homology with coefficient $M$ is the homology theory on $\eqsch$ associated to
the homology functor $\KCGM -$ (see Definition \ref{def.homfunc2}).
By definition, for $(X,G)\in \eqsch$, 
$$
\KHGM a X = H_a(\KCGM X)= E^2_{a,0}(X,G;M)\quad (a\in \bZ_{\geq 0}).
$$
In case $G=e$, $\KHGM a X$ is simply denoted by $\KHM a X$. 
\end{defn}

By Proposition \ref{eqhom.prop1}(1)
we have an edge homomorphism for each integer $a\geq 0$:
\begin{equation}\label{eqKH.edgehom}
\edgehom a X: H^{\et}_{a}(X,G;M) \to \KHGM a X
\end{equation}
which is viewed as a map of homology theories (see Definition \ref{def.eqhom}) on $\eqsch$.

\bigskip

Now we assume that $G_k$ acts trivially on $M$ and 
construct a map of homology theories on $\eqsch$
\[
\gamma_M: \{\KHM a -\}_{a\in \bZ} \to \{\HWM a -\}_{a\in \bZ}
\]
First note
$H^{\et}_0(\Spec(k);M)\cong M$, and hence
\[
\KCM {\Spec(k)}=M[0],
\]
where $\Spec(k)$ is viewed as an object of $\eqsch$ with the trivial action of $e$. 
For $(X,G)\in \cSpr$, the natural map $\pi:X\to \Spec(k)$ induces a map of complex
\begin{equation*}
\pi_* :\KCGM X \to \KCM {\Spec(k)}=M[0].
\end{equation*}
For $(X,G)\in \cSeq$, we get a natural map (see \eqref{examweighthom.eq1})
\begin{equation}\label{eqKhom.3}
\gamma_M :\KCGM X \to M^{\Xcd 0/G}[0] \simeq F^W_M(X,G) 
\end{equation}
by taking the sum of the above maps for the $G$-orbits of connected components of $X$. 
This gives us a map of prehomology theories
\begin{equation}\label{eqKhom.4}
\gamma_M: \KCM - \to F^W_M 
\end{equation}

\begin{theo}\label{thm.weighthom}
Assume the condition $(\bigstar)$ of Theorem \ref{thm.Whom}.
There exists a map of homology theories on $\eqsch$:
\[
\gamma_M: \{\KHM a -\}_{a\in \bZ} \to \{\HWM a -\}_{a\in \bZ}
\]
such that for $(X,G)\in \cSeq$, it coincides with the map induced from \eqref{eqKhom.3}.
\end{theo}
\begin{proof}
If $k$ is perfect, this is an immediate consequence of 
Proposition \ref{thm.KHdescent} and Construction \ref{funconst}(4) with Lemma \ref{PHIF.lem}(1).
The general case follows from the above case thanks to Theorem \ref{thm.Whom}$(vi)(a)$
and Proposition \ref{eqKH.prop2}(2).
\end{proof}

\bigskip

\section{Proof of McKay principle for weight homology}\label{proof}

\noindent
Let the notation be as in \S\ref{eqKhom} and assume that $G_k$ acts trivially on $M$. 
Recall that $\Lam$ denotes either $\bZ$ or $\bZ[1/p]$
where $p$ is the exponential characteristic of $k$.
Write $\Linfty=\bQ/\Lam$ and $\Ln=\Lam/n\Lam$ for an integer $n>0$.
Concerning the map of homology theories on $\eqsch$:
\[
\gamma_M: \{\KHM a -\}_{a\in \bZ} \to \{\HWM a -\}_{a\in \bZ},
\]
we have the following result, which is one of the main results of this paper.

\begin{theo}\label{mainthm1} 
Assume that $k$ is a purely inseparable extension of a  finitely
generated field.
Assume the condition $(\bigstar)$ of Theorem \ref{thm.Whom}.
Then $\gamma_M$ is an isomorphism for $M=\Linfty$. In particular, for $(X,G)\in \cSeq$
we have
\[
\KHGLinf a X= \left\{ \begin{array}{lr}
(\Linfty)^{\Xcd 0/G} & \text{ for } a=0 \\
0 & \text{ for } a\ne 0 \end{array} \right. 
\]
where $\Xcd 0$ is the set of the generic points of $X$.
If $k$ is finite, the same holds by replacing $\Linfty$ by $\Ln$ for any integer $n>0$.

\end{theo}

Theorem \ref{thm.Whomgraph} and Proposition~\ref{prop.bounded} will be deduced from from Theorem \ref{mainthm1}.
Theorem \ref{mainthm1} will be deduced from the following result quoted from 
\cite{JS1} Theorem 3.8 and \cite{KeS} Theorem 3.5.

\begin{theo}\label{thm.KeS}
Assume that $k$ is a purely inseparable extension of a finitely
generated field.
Assume the condition $(\bigstar)$ of Theorem \ref{thm.Whom}.
Then, for a smooth projective scheme $X$ over $k$, we have $\KHLinf a X=0$ for $a\not=0$.
If $k$ is finite, the same holds by replacing $\Linfty$ by $\Ln$ for any integer $n>0$.
\end{theo}

Note that Theorem \ref{mainthm1} and Proposition \ref{eqKH.prop1} imply the following extension of 
Theorem \ref{thm.KeS} to a singular case.

\begin{coro}\label{mainthm1.cor} 
Let the assumption be as in Theorem~\ref{thm.KeS}.
For $(X,G)\in \cSeq$, we have $\KHLinf a {X/G}=0$ for $a\not=0$.
If $k$ is finite, the same holds by replacing $\Linfty$ by $\Ln$ for any integer $n>0$.
\end{coro}
\bigskip

For the proof of Theorem \ref{mainthm1}, we prepare some lemmas.

\begin{lem}\label{lem.proof0}
For $X\in \cSeq$, the map \eqref{eqKhom.3} induces an isomorphism
$$
\KHGM 0 X  \stackrel{\sim}{\to} H^W_0(X,G;M) .
$$
\end{lem}
\begin{proof}
We may suppose that $X$ is primitive.
By Proposition \ref{eqhom.prop1}(4), the edge homomorphism
\[
\epsilon^0_X: H^{\et}_0(X,G;M) \to \KHGM 0 X
\]
is an isomorphism. Thus the assertion follows from the fact that the map
\[
H^{\et}_0(X,G;M) \to H^{\et}_{0}(\Spec(k),M)
\]
is an isomorphism, which is
easily seen from the Definition \ref{eqhom.def1}. 
\end{proof}
\medbreak

\begin{lem}\label{eqKhom.lem3}
Fix a finite group $G$ and assume $\KHGM i Z=0$ for all $(Z,G)\in
\cSeq$ and for all $i>0$. Then 
\begin{equation}\label{eq1}
\gamma_{M} : \KHGM a X \to \HWGM a X
\end{equation}
is an isomorphism for all $(X,G)\in \schG$ and for all $a\in \bZ$.
\end{lem}
\begin{proof}
By Lemma \ref{lem.proof0}, the assumption implies that \eqref{eq1} is an isomorphism 
for all $(X,G)\in \cSG$ and for all $a\in \bZ$.
Then the lemma follows from Construction \ref{funconst}(4).
For convenience of the readers, we give a proof.
We suppress the coefficient $M$ from the notation.
First assume that $X$ is proper over $k$. Take a $\Lam$-admissible hyperenvelope 
$f: (X_\bullet,G) \to (X,G)$ in $\schG$ such that $(X_a,G)\in \cSG$ for all $a\ge 0$.
By Construction \ref{funconst}(2) and Theorems \ref{thm.KHdescent}, we have spectral sequences
$$
\begin{aligned}
&^KE^1_{a,b} =\KHGXa b \Rightarrow \KHGX {a+b},\\
&^WE^1_{a,b} =\HWGXa b \Rightarrow \HWGX {a+b},\\
\end{aligned}
$$
and $\gamma_M$ induces a map of spectral sequences.
By the assumption the map on the $E_1$ terms are all isomorphism which implies
the desired isomorphism.

\medbreak

In case $X$ is not proper, take an open immersion $j:X \hookrightarrow \Xb$ in $\schG$
with $\Xb$ proper over $k$. Putting $Y=\Xb-X$, we have a commutative diagram 
with exact horizontal sequences
\begin{small}
$$
\xymatrix@C=20pt{
\KHG i Y \ar[r] \ar[d]_{\gamma}  &   \KHG i {\Xb}  \ar[r] \ar[d]_{\gamma}  & \KHG i {X}
\ar[r] \ar[d]_{\gamma}  & 
\KHG {i-1} Y \ar[r] \ar[d]_{\gamma}  & \KHG {i-1} {\Xb} \ar[d]_{\gamma}  \\
\HWG i Y \ar[r]&  \HWG i {\Xb}  \ar[r] &  \HWG i {X} \ar[r] &  \HWG
{i-1} Y \ar[r] &
\HWG {i-1} {\Xb}   }
$$
\end{small}
Since $\Xb$ and $Y$ are proper over $k$, all vertical maps except the middle are isomorphisms
and so is the middle one. This completes the proof. 
\end{proof}
\bigskip

{\it Proof of Theorem \ref{mainthm1}:}
By Theorem~\ref{thm.Whom}$(vi)$ and Proposition \ref{eqKH.prop2}(2),
we may assume that $k$ is the perfection of a finitely generated field.
Let $M=\Ln$ for an integer $n>0$ if $k$ is finite, and let $M=\Linfty$ otherwise.
For $(X,G)\in \eqsch$, let $\pi:(X,G)\to (X/G,e)$ be the induced map in $\eqsch$.
Consider the commutative diagram
$$
\begin{CD}
\KHGM a X @>{\simeq}>{\pi_*}> \KHM a {X/G}\\
@VV{\gamma}V @V{\simeq}V{\gamma}V \\
\HWGM a X @>>{\pi_*}> \HWM a {X/G} \\
\end{CD}
$$
The map $\pi_*$ in the upper row is an isomorphism by Proposition \ref{eqKH.prop1}.
The map $\gamma$ on the right hand side is an isomorphism by Theorem \ref{thm.KeS} and 
Lemma \ref{eqKhom.lem3} with $G$ the trivial group. In case $(X,G)\in \cS_{eq/k}$, $\HWGM a X=0$ for $a\not=0$ by 
Theorem \ref{thm.Whom}$(iv)$. Thus we get $\KHGM a X=0$ from the diagram.
By Lemma \ref{eqKhom.lem3}, this implies that $\graphhom a X$ is an isomorphism 
for all $a\in \bZ$ and for all $(X,G)\in \eqsch$. 
This completes the proof of Theorem \ref{mainthm1}.
$\square$
\bigskip

Now we deduce Theorem~\ref{thm.Whomgraph} from Theorem \ref{mainthm1}. 
By Theorem~\ref{thm.Whom}$(vi)$, we may assume that $k$ is the perfection of a finitely generated field. 
By the universal coefficient spectral sequence in Theorem~\ref{thm.Whom}$(iv)$, we may assume $M=\Lam$.
Theorem~\ref{thm.Whom}$(v)$ shows that  $ \HWGL a X $ and $ \HWL a {X/G}$ 
are finitely generated $\Lam$-modules.
The subcategory ${\it Modf}_{\Lam} \subset {\it Mod}_{\Lam}$ of finitely generated $\Lam$-modules is a Serre-subcategory, so we can talk about the
quotient category ${\it QMod}_\Lam={\it Mod}_{\Lam} / {\it Modf}_{\Lam}$.
The long exact homology sequence for weight homology associated to
\begin{equation}\label{eqshortex}
0\to \Lam \to \bQ \to \Linf\to 0 
\end{equation}
gives us isomorphisms in ${\it QMod}_\Lam$:

$$
\HWGQ a X \simeq \HWGLinf a X\;\;\;  \text{ and }\;\;\;  
\HWQ {a} {X/G} \simeq \HWLinf {a}{X/G}.
$$
By Theorem~\ref{mainthm1} and Proposition~\ref{eqKH.prop1} we have isomorphisms in ${\it Mod}_{\Lam}$
\[
\HWGLinf a X  \stackrel{\sim}{\leftarrow}  \KHGLinf a X  \stackrel{\sim}{\to}  
\KHLinf a {X/G}  \stackrel{\sim}{\leftarrow} \HWLinf a {X/G}.
\]
Putting this together we get that $\pi_*:(X,G)\to (X/G,e)$ induces an isomorphism 
in ${\it QMod}_\Lam$:
\begin{equation}\label{eqaltenQ}
\HWGQ a X  \simeq \HWQ a {X/G}.
\end{equation}
Lemma~\ref{lem.equalQLam} below implies that \eqref{eqaltenQ} is a true isomorphism of $\Q$-modules.

\begin{lem}\label{lem.equalQLam}
If a morphism of $\Q$-modules $\phi:A \to B$ becomes an isomorphism
in the category ${\it QMod}_{\Lam}$ then $\phi$ is already a true isomorphism of $\Q$-modules.
\end{lem}

We deduce Theorem~\ref{thm.Whomgraph} with $M=\Lam$ using the five-lemma from the commutative diagram with exact rows associated to the sequence \eqref{eqshortex}
\begin{small}
\[
\xymatrix@C=10pt{
\HWGQ {a+1} X \ar[r] \ar^{\wr}[d] &   \HWGLinf {a+1} X \ar[r] \ar^{\wr}[d] &    
\HWGL {a} X \ar[r] \ar[d] & \HWGQ {a} X \ar[r] \ar^{\wr}[d] & \HWGLinf {a} X\ar^{\wr}[d]\\
\HWQ {a+1} {X/G} \ar[r] & \HWLinf {a+1}{X/G}\ar[r] &  \HWL {a}{X/G} \ar[r] & 
\HWQ {a}{X/G} \ar[r] & \HWLinf {a}{X/G}.
}
\]
\end{small}
This completes the proof of Theorem~\ref{thm.Whomgraph}.
\medbreak

The following theorem is an analogue of Theorem \ref{thm.Whomgraph} for a radicial morphism.

\begin{theo}\label{thm.Whom.radicial}
Assume $\ch(k)=p>0$ and that the condition $(\bigstar)$ of Theorem \ref{thm.Whom} holds.
Let $M$ be a $\Lam$-module.
A finite radicial surjective morphism $\pi: X' \to X$ in $\sch$ induces an isomorphism
\[
\pi_* : \HWM a {X'} \isom \HWM a {X}\qfor a\in \bZ.
\]
\end{theo}
\begin{proof}
By the same argument as the proof of Theorem~\ref{thm.Whomgraph},
we are reduced to the case $M=\Linfty=\bQ/\Lam$. 
Then the theorem follows from Theorem \ref{mainthm1} and Proposition \ref{eqKH.prop2}.
\end{proof}

\medskip

{\it Proof of Proposition \ref{prop.bounded}}:
We may assume that the base field $k$ is finitely generated. For $M=\Linf$ we have by Theorem \ref{mainthm1}
\[
\gamma_{\Lambda_\infty} : \KHGLinf a X  \stackrel{\simeq}{\longrightarrow}  \HWGLinf a X.
\]
Thus we get $\HWGLinf a X=0$ for $a>\dim(X)$ since $\KHGLinf a X$ vanishes in degrees $a>\dim(X)$ 
by definition. By a simple devissage, this implies 
\begin{enumerate}
\item[$(*)$]
$\HWGM a X=0$ for $a>\dim(X)$ and $M$ torsion.
\end{enumerate}
For a finitely generated $\Lam$-module $M$, we consider the exact sequences
\[
0\to M/M_{tor} \to M\otimes_\bZ\bQ  \to M\otimes\qz \to 0,
\]
\[
0\to M_{tor} \to M \to M/M_{tor} \to 0,
\]
where $M_{tor}$ is the torsion part of $M$.
By $(*)$ the first sequence implies
\[
H^W_a(X,G;M/M_{tor}) \simeq H^W_a(X,G;M/M_{tor})\otimes_{\bZ}\bQ \qfor a>\dim(X),
\]
which implies $H^W_a(X,G;M/M_{tor})=0$ for $a>\dim(X)$ since $H^W_a(X,G;M/M_{tor})$ is 
a finitely generated $\Lam$-module. By $(*)$ we deduce from the second sequence 
that $H^W_a(X,G;M)=0$ for $a>\dim(X)$. 
This completes the proof of Proposition~\ref{prop.bounded}.
$\Box$
\bigskip

\section{Dual complexes and cs-coverings}\label{csfundgroup}
\bigskip

The dual complex of a simple normal crossing divisor is a special kind of $CW$-complex as
sketched in the introduction (called $\Delta$-complex in  \cite[Section 2.1]{Hat}).
It describes the configuration of the irreducible components of the divisor (see \cite{St}, \cite{Pa} and \cite{ABW}). 

Let $E$ be a locally noetherian scheme and $\{E_v\}_{v\in K}$ be the set of irreducible components of $E$. 
For a finite subset $\sigma=\{v_0,\dots,v_k\}\subset K$, put
\[
E_\sigma:=E_{v_0}\cap\cdots\cap E_{v_k}.
\]
Let $\{E_\alpha\}_{\alpha\in K_\sigma}$ be the set of connected components of $E_\sigma$.
In case $E_\sigma=\varnothing$, $K_\sigma=\varnothing$ by convention.

\medbreak

To a locally noetherian scheme $E$, we associate its dual complex $\Gamma=\dualG E$ as follows.
Fix a linear ordering of $K$.
The set of vertices of $\dualG E$ is identified with $K$.
For a finite subset $\sigma=\{v_0,\dots,v_k\}\subset K$ with $v_0 < \cdots < v_k$, put
\[
\Delta_\sigma=\{\underset{0\leq i\leq k}{\sum}\; t_i v_i\;|\; t_i\geq 0,\; 
\underset{0\leq i\leq k}\sum t_i= 1\}\;\subset \bR^K.
\]
To each $\alpha\in K_\sigma$,
is associated a $k$-simplex $\Delta_{\sigma,\alpha}$ of $\dualG E$ which is a copy of $\Delta_\sigma$.
The interiors of these simplices are disjoint in $\dualG E$ and we have 
\[
\dualG E =\underset{\sigma\subset K}{\coprod}\underset{\alpha\in K_\sigma}{\coprod} 
\;\Delta_{\sigma,\alpha}/\sim,
\]
where for finite subsets $\tau\subset \sigma\subset K$ and for $\alpha\in K_\sigma$ and 
$\beta\in K_\tau$ such that $E_\alpha\subset E_\beta$,
$\Delta_{\tau,\beta}$ is identified with a face of $\Delta_{\sigma,\alpha}$ by the
inclusion $\Delta_{\tau}\hookrightarrow \Delta_{\sigma}$ induced by $\tau\hookrightarrow
\sigma$.

Clearly, as a topological space $\Gamma(E)$ does not depend on the ordering of $K$.
\medbreak

Let $f:E\to E'$ be a morphism of locally noetherian schemes.
Let $\{E'_w\}_{w\in K'}$ be the set of irreducible components of $E'$.
Choosing $\phi(v)\in K'$ for each $v\in K$ such that $f(E_v)\subset E'_{\phi(v)}$, we get
a continuous map
\[
\phi: K \to K'.
\]
For each $\alpha\in K_\sigma$,
there is a unique $\phi(\alpha) \in K'_{\phi(\sigma)}$ 
such that $f(E_\alpha) \subset E'_{\phi(\alpha)}$. We define an affine map
\[
\phi: \Delta_{\sigma,\alpha} \to \Delta_{\phi(\sigma),\phi(\alpha)}\;;\;
\underset{0\leq i\leq k}{\sum}\; t_i v_i\; \to \underset{0\leq i \leq k}{\sum}\; t_i \phi(v_i)
\quad (t_i\geq 0,\; \sum t_i= 1).
\]
These maps glue to induce a map of $\Delta$-complexes
\[
\phi: \dualG E \to \dualG {E'}.
\]

\begin{lem}\label{lem1.dualG}
Up to homotopy, $\phi$ does not depends on the choice of $\phi: K\to K'$.
\end{lem}
\begin{proof}
Pick $v_0\in K$ and let $\phi(v_0)=w_0\in K'$.
Assume $f(E_{v_0})\subset E'_{w_0'}$ for $w_0'\in K'$ with $w_0'\not=w_0$.
Let $\psi: \dualG E \to \dualG {E'}$ be the map induced by the map $\psi:K\to K'$ defined as
$\psi(v)=\phi(v)$ if $v\not=v_0$ and $\psi(v_0)=w_0'$.
It suffices to show that $\psi$ is homotopic to $\phi$.
Let $\sigma=\{v_0,\dots,v_k\}\subset K$ and $\alpha\in K_\sigma$. The assumption implies
\[
f(E_\alpha)\subset f(E_\sigma)\subset E'_\theta=
E'_{w_0}\cap E'_{w_0'}\cap E_{\phi(v_1)}\cap \cdots \cap E_{\phi(v_k)},
\]
where 
\[
\theta=\phi(\sigma)\cup \psi(\sigma)=\{w_0,w_0',\phi(v_1),\dots,\phi(v_k)\} \subset K'.
\]
There is a unique connected component $E'_\gamma$ of $E'_\theta$ such that 
$f(E_\alpha)\subset E'_\gamma$. We have $E'_\gamma \subset E'_{\phi(\alpha)} \cap E'_{\psi(\alpha)}$
and hence $\Delta_{\phi(\sigma),\phi(\alpha)}$ and $\Delta_{\psi(\sigma),\psi(\alpha)}$ are 
faces of $\Delta_{\theta,\gamma}$. We define a map
\[
H_{\sigma,\alpha}: [0,1]\times \Delta_{\sigma,\alpha} \to \Delta_{\theta,\gamma}\;;\;
(r,\underset{0\leq i\leq k}{\sum}\; t_i v_i) \; \to 
t_0(r w_0 +(1-r) w_0') + \underset{1\leq i \leq k}{\sum}\; t_i \phi(v_i).
\]
which gives a homotopy from $\phi_{|\Delta_{\sigma,\alpha}}$ to $\psi_{|\Delta_{\sigma,\alpha}}$.
The maps $H_{\sigma,\alpha}$ for all $\sigma$ and $\alpha$ glue to give a desired homotopy
$H: [0,1]\times \dualG E \to \dualG {E'}$.
\end{proof}
\bigskip

\begin{lem}\label{lem.Gstrict}
Let $E$ be a $G$-strict simple normal crossing divisor over a field as in Definition~\ref{def.Gstrict}. 
There is a canonical homeomorphism
\[
\dualG {E/G} \isom \dualG E/G,
\]
of topological spaces.
\end{lem}
\begin{proof}
The proof is straightforward and left to the readers.
\end{proof}

\bigskip

To investigate fundamental groups of dual complexes, we need a more geometric interpretation.
In the following we introduce the technique of covering spaces of non-path-connected
topological spaces, which we apply to the Zariski topological space underlying a locally
noetherian scheme. 

\newcommand\cscov{{\rm{Cov}}}
\newcommand\Sets{{\rm{Sets}}}

\begin{defi}
For a connected topological space $S$ 
let $\cscov(S)$ be the category of topological coverings of $S$.
For a point $s_0\in S$, let $\pics {S,s_0}$ be the automorphism group of the fiber functor
\[
\cscov(S) \to \Sets\;;\; S' \mapsto S'\times_S s_0.
\]
In what follows we omit the base point $s_0$ from the notation. We call $\pics S$ the
completely split fundamental group.
\end{defi}

If $S$ is a connected scheme we write $\pics S $ for the cs-fundamental group of the underlying
Zariski topological space of $S$. 

\begin{lem}\label{unicovsp}
If the connected topological space $S$ admits a covering by open subsets $\mathcal U =(U_i)_{i\in I}$ such that
 $\cscov (U_i)$ consists only of the trivial coverings for all $i\in I$, then there is a universal connected covering
$S'\to S$. We have  $\pics S= {\rm Aut} (S'/S)$.
\end{lem}

\begin{proof}
Fix a linear ordering of $I$.
One associates to the open covering $\mathcal U$ a $\Delta$-complex $\Gamma(\mathcal U)$
in analogy to the construction above. Its $k$-simplices are the connected
components of $U_{v_0} \cap \cdots \cap U_{v_k}$ for $v_0<\cdots < v_k \in I$. There is a
canonical equivalence $\cscov(S) \cong \cscov(\Gamma(\mathcal U))$.
In fact objects of both sides are described in terms of the following data: A family of
sets $(W_i)_{i\in I}$ and a family of isomorphisms $w_{\alpha}:W_i \isom W_j$ for each connected
component $\alpha$ of $U_i \cap U_j$ ($i<j$) which satisfy cocycle conditions (one for each
connected component of $U_i\cap U_j \cap U_k$ with $i<j<k$).
 Since
$\Gamma(\mathcal U)$ is a $CW$-complex, it is locally path-connected
and locally simply connected and therefore by \cite[Thm.~1.38]{Hat} it has a universal
covering space.
\end{proof}
\smallskip

The assumption of Lemma~\ref{unicovsp} is satisfied in particular for the topological space underlying a locally
noetherian scheme.

\begin{prop}\label{thm.pics} Let $X,Y$ be connected locally noetherian schemes and let $f:X\to Y$ be
  a proper surjective morphism with connected fibers. Then the induced map
\[
f_* : \pics X \to \pics Y
\]
is surjective.
\end{prop}

\begin{proof}
Let $Y'\to Y$ be the universal connected covering. Clearly $Y'$ has a canonical scheme structure. We have to show that $X':=X\times_Y Y'$ is connected.
In fact the fibers of $X'\to Y'$ are connected and the map is closed and surjective. Therefore different connected
components of $X'$ map onto different connected components of $Y'$, but
there is just one of the latter.
\end{proof}

\begin{prop}\label{prop.cspi1}
For a connected locally noetherian scheme $E$, there is a natural isomorphism
\[
\gamma_E: \pics E \isom \pi_1(\dualG E).
\]
\end{prop}
\begin{proof}
An irreducible scheme has only trivial cs-topological coverings. This implies that there is a
canonical equivalence $\cscov (E) \cong \cscov(\Gamma(E))$ by an argument similar to the proof of Lemma~\ref{unicovsp}.
\end{proof}
\bigskip

\begin{coro}\label{csfundgroup.thm}
Let $E$ be a $G$-strict simple normal crossing divisor as in Definition \ref{def.Gstrict}.
Then we have a natural isomorphism
\[
  \pi_1^{cs}(E/G) \isom \pi_1(\dualG {E}/G).
\]
\end{coro}

\begin{proof}
One combines the isomorphism $\gamma_{E/G}$ of Proposition~\ref{prop.cspi1} with the
isomorphism of Lemma~\ref{lem.Gstrict}.
\end{proof}
\bigskip

\section{McKay principle for homotopy type of dual complexes}\label{MacKayhomtype}

\medskip

\renewcommand\EYb{\overline{E}_Y}
\renewcommand\pihat{\widehat{\pi}}
\renewcommand{\wh}[1]{\widehat{#1}}
\renewcommand{\wht}[1]{\widehat{\tilde{#1}}}
\renewcommand{\dualG}[1]{\Gamma(#1)}

Let the notation be as \S\ref{eqWhom}. 
We assume $\ch(k)=0$ or canonical resolution of singularities in the sense of \cite{BM} holds over $k$.
Let $(X,G)\in \schG$ with $X$ smooth and $\pi:X\to X/G$ be the projection.
Fix a closed reduced subscheme $S\subset X/G$ which is projective over $k$ and contains the singular locus $(X/G)_{sing}$ of $X/G$.
Let $T=\pi^{-1}(S)_{red}$ be the reduced part of $\pi^{-1}(S)$. 
Assume that we are given the following datum:
\begin{itemize}
\item
a proper birational morphism $g:\tY \to X/G$ such that $\tY$ is smooth, $E_S=g^{-1}(S)_{red}$ is a simple normal crossing divisor on $\tY$ and $g$ is an isomorphism over $X/G-S$.
\item
a proper birational $G$-equivariant morphism $f:\tX \to X$ in $\schG$ such that $\tX$ is smooth, $E_T=f^{-1}(T)_{red}$ is a $G$-strict simple normal crossing divisor on $\tX$ 
(cf. Definition \ref{def.Gstrict}) and $f$ is an isomorphism over $X-T$.
\end{itemize}
\[
\begin{CD}
\pi^{-1}(S)_{red}=\; @. T @>>>  X @<{f}<< \tX @<<< E_T = f^{-1}(T)_{red}   \\
 @. @VVV @VV{\pi}V\\
 @. S  @>>> {X/G} @<{g}<<   \tY  @<<< E_S=g^{-1}(S)_{red} 
\end{CD}
\]
\bigskip

Note that we do not assume that there exists a morphism $E_T \to E_S$.
By definition $G$ act on $\Gamma(E_T)$ and we can form the topological space $\Gamma(E_T)/G$.

\begin{theo}\label{thm.MPht} 
In the homotopy category of $CW$-complexes, there exists a canonical map
\[
\phi: \Gamma(E_T)/G \to \Gamma(E_S)
\]
which induces isomorphisms on the homology and fundamental groups:
\[
H_a(\Gamma(E_T)/G) \isom  H_a(\Gamma(E_S))\qfor \forall a\in \bZ,
\]
\[
\pi_1(\Gamma(E_T)/G) \isom  \pi_1(\Gamma(E_S)).
\]
If $\Gamma(E_T)/G$ is simply connected, then $\phi$ is a homotopy equivalence.
\end{theo}
\bigskip

The following corollaries are immediate consequences of Theorem \ref{thm.MPht}.

\begin{coro}\label{thm.homequiv} 
If $\dualG {E_T}/G$ is contractible, then $\dualG {E_S}$ is contractible.
\end{coro}

\begin{coro}\label{thm.homequiv2} 
If $T$ is smooth (e.g. $\dim(T)=0$ which means that $(X/G)_{sing}$ is isolated), 
then $\dualG {E_S}$ is contractible.
\end{coro}

Indeed, if $T$ is smooth, one can take $\tX$ to be the blowup of $X$ along $T$ and $\Gamma(E_T)/G$ consists only of $0$-simplices.

\begin{coro}\label{thm.homequiv3}
Let $A$ be a complete regular local ring containing $\Q$ and let $G$ be a finite group acting on
$A$. Set $X=\Spec (A)$ and assume that $X/G$ has an isolated singularity $s\in X/G$. Let $g:\tilde Y\to
X/G$ be a proper morphism such that $g$ is an isomorphism outside $s$ and $E_s=g^{-1}(s)_{red}$
is a simple normal crossing divisor in the regular scheme $\tilde Y$. Then the topological
space $\Gamma(E_s)$
is contractible.
\end{coro}

\begin{proof}
Note that $Y=X/G=\Spec(A^G)$ is automatically noetherian, $X\to Y$ is finite and $Y$ is complete, because we are in characteristic
zero \cite[Sec.\ 1.2]{Mum}.
 One observes that the morphism
$X\to Y$ is flat and locally complete intersection outside $s$, see \cite[Thm.\ 46, Thm.\
36(4)]{Mat}. 

By algebraization of isolated singularities \cite[Thm.\ 3.8]{Art} the
complete local scheme
$Y$ results from the completion at an isolated singularity $s'\in Y'$ of a scheme $Y'$ of finite
type over the residue field $k$ of $A^G$. By \cite[Thm.\ 4]{Elk} the morphism $X\to Y$
algebraizes to a morphism $X'\to Y'$ together with the group action $G$ on $X'/Y'$, after possibly replacing $Y'$ by an \'etale
neighborbood of $s'$. Again after shrinking $Y'$ around $s'$ and using Artin approximation \cite{Art} one finds a proper morphism
$\tilde Y' \to Y'$ such that the
base change $\tilde Y' \times_{Y'} Y$ coincides with $\tilde Y \to Y$ to some arbitrary
high order with respect to powers of the maximal ideal of $A^G$.
 By choosing the order large one can make sure that $\tilde Y' \to Y'$ is an
isomorphism outside $s'$, that $Y'$ is regular and that the reduced preimage $E_{s'}$ of $s'$ in $\tilde Y'$
is a simple normal crossing divisor. 

It follows from Corollary~\ref{thm.homequiv2} that $\Gamma( E_{s'}) =\Gamma(E_{s}) $ is contractible.
\end{proof}

\bigskip
\noindent
{\it Proof of Theorem \ref{thm.MPht}}
\medbreak
The last assertion follows from the first using the relative version of the Hurewicz theorem (cf. \cite[4.33]{Hat}).
Choose a $G$-equivariant proper birational morphism $h:\tX' \to \tX$ such that
\begin{itemize}
\item
$\tX'$ is smooth, $E'_S=h^{-1}(E_T)_{red}$ is a $G$-strict simple normal crossing divisor on $\tX'$ 
and $h$ is an isomorphism over $\tX-E_T$,
\item
there exists a $G$-equivariant proper morphism $\tX' \to \tY$.
\end{itemize}
Thus we get a commutative diagram 
\[
\xymatrix{
T \ar[d] &\ar[l] E_T\ar[d] &\ar[l] E'_T\ar[d]\\
 X \ar[d]^\pi & \ar[l]  \tX   & \ar[l]_h \tX'  \ar[ld] &   \\
{X/G} &  \ar[l] \tY & \\
}
\]
and a $G$-equivariant map $\rho: E'_T\to E_T$. 

\begin{claim}\label{MP.claim1}
 $\rho$ induces a $G$-equivariant homotopy equivalence $\Gamma(E'_T) \underset{h.e.}{\simeq} \Gamma(E_T)$
and hence a homotopy equivalence $\Gamma(E'_T)/G \underset{h.e.}{\simeq} \Gamma(E_T)/G$.
\end{claim}

The claim is a consequence of the following general fact:

\begin{theo}\label{thm.hidual}
Let $G$ be a group.
Let $\phi: X' \to X$ be a $G$-equivariant proper morphism of smooth projective $G$-schemes over $k$ such that
$\phi$ is an isomorphism over a dense open subset $U\subset X$ and such that $E=X-U$
(resp. $E'=X'-\phi^{-1}(U)$) is a $G$-strict simple normal crossing divisor on 
$X$ (resp. $X'$). Then the induced map $\Gamma(E')\to \Gamma(E)$ is a $G$-equivariant homotopy equivalence.
\end{theo}

The case $G=e$ is Stepanov's theorem \cite{St} and its generalizations \cite[Theorem 1.1]{Pa} and \cite[Theorem 7.5]{ABW}.
The $G$-equivariant version is proved by the exactly same argument by replacing the weak factorization theorem of Wlodarczyk \cite{Wl} with its equivariant version 
\cite[Theorem 0.3.1 and Remark (2)]{AKMW}.
$\square$

\bigskip

By Claim \ref{MP.claim1} we may assume that there exist a morphism $\phi: \tX/G\to \tY$.
Take a proper birational morphism $h:\tY' \to \tX/G$ such that $\tY'$ is smooth, $E'_S=h^{-1}(E_T/G)_{red}$ is a simple normal crossing divisor on $\tY'$ 
and $h$ is an isomorphism on $\tY'- E'_S$. Thus we get a commutative diagram:
\begin{equation}\label{MP.eq2}
\begin{matrix}
\pi^{-1}(S)=\hskip -10pt&T &\gets & E_T \\
 & \downarrow & &\downarrow \\&
  X &\lmapo{f}& \tX    \\
 & \downarrow\rlap{$\pi$} &&   &\searrow  \\
 & {X/G} &\lmapo{g}&   \tY  &\lmapo{\phi} & \tX/G  &\lmapo{h} & \tY'\\
 & \uparrow & &\uparrow& &\uparrow & &\uparrow& \\
 & S &\gets & E_S &\gets & E_T/G &\gets & E'_S
\end{matrix}
\end{equation}
From the diagram we get maps of $CW$-complexes:
\[
\Gamma(E'_S) \rmapo{h_*} \Gamma(E_T)/G \rmapo{\phi_*} \Gamma(E_S).
\]
The composite of these maps is a homotopy equivalence by Theorem \ref{thm.hidual}.
Hence the assertion on the fundamental groups follows if one shows the surjectivity of the induced maps
\[
\pi_1(\Gamma(E'_S)) \rmapo{h_*} \pi_1(\Gamma(E_T)/G) \rmapo{\phi_*} \pi_1(\Gamma(E_S)).
\]
Noting $\tY'\to \tX/G \to \tY$ have geometrically connected fibers by Zariski's main theorem, the assertion follows from
Proposition \ref{prop.cspi1}, Corollary \ref{csfundgroup.thm} and Proposition \ref{thm.pics}.
\medbreak

It remains to show the assertion on the homology groups. Recall that we are given a
commutative diagram 
\begin{equation}\label{MP.eq3}
\begin{CD}
\pi^{-1}(S)_{red}=\; @. T @>>>  X @<{f}<< \tX @<<< E_T = f^{-1}(T)_{red}   \\
 @. @VVV @VV{\pi}V @VV{\tilde{\pi}}V\\
 @. S  @>>> {X/G} @<{g}<<   \tY  @<<< E_S=g^{-1}(S)_{red} 
\end{CD}
\end{equation}
where the assumption is the same as in the beginning of this section.
Note that we are now given $\tilde{\pi}:\tX\to \tY$ which extends $\pi: X\to X/G$.
It induces a map $\phi: E_T/G \to E_S$ and we are ought to show that it induces isomorphisms on the homology groups
\[
H_a(\Gamma(E_T)/G) \isom H_a(\Gamma(E_S))\quad (a\in \bZ).
\]
Since $S$ (and hence $T$) is assumed projective over $k$, we may choose $X$, $\tX$ and $\tY$ 
projective over $k$. By Example \ref{Whom.ex},
the homology groups of the dual complexes are described as the weight homology groups:
\[
H_a(\dualG {E_T}/G,M)\simeq \HWGM a {E_T},\quad 
H_a(\dualG {E_S},M)\simeq \HWM a {E_S},
\]
where $M$ is any coefficient module.
Thus the desired assertion follows from the following claim which holds over a field $k$ of arbitrary characteristic.

\begin{claim}\label{thm.fundgroup0} 
Let the assumption be as \eqref{MP.eq3} and assume that $X$, $\tX$ and $\tY$ are projective over $k$.
Assume the condition $(\bigstar)$ of Theorem \ref{thm.Whom}.
For a $\Lam$-module $M$, the map $(E_T,G)\to (E_S,e)$ in $\eqsch$ induces an isomorphism
\[
\HWGM a {E_T}\simeq \HWM a {E_S}.
\]
\end{claim}
\begin{proof}
We have a commutative diagram with exact rows
\begin{small}
\[
\xymatrix@C=20pt{
 \HWGM {a+1} \tX \ar[d] \ar[r] & \HWGM {a+1} {\tX-E_T} \ar[r] \ar[d]^{\simeq} & 
 \HWGM {a} {E_T} \ar[d] \ar[r] &  \HWGM {a} \tX \ar[d]  \\
 \HWM {a+1} {\tY} \ar[r]  &  \HWM {a+1} {\tY-E_S} \ar[r] &   \HWM {a} {E_S} \ar[r]  &   \HWM {a} \tY .
}
\]
\end{small}
The second vertical isomorphism follows from Theorem \ref{thm.Whomgraph} 
noting $(\tX-E_T)/G \cong \tY-E_S$.
By Theorem \ref{thm.Whom}$(i)$ this implies the desired isomorphism of the claim.
\end{proof}

\bigskip

\end{document}